\documentclass[a4paper, 12pt]{article} 
\usepackage{amssymb,amsmath,latexsym,color,graphicx,stmaryrd}
\usepackage{hyperref} 

\newtheorem{dref}{Definition}[section] \newtheorem{lemma}[dref]{Lemma}
\newtheorem{theo}[dref]{Theorem} \newtheorem{prop}[dref]{Proposition}
\newtheorem{remark}[dref]{Remark} \newtheorem{ex}[dref]{Example}
\newtheorem{cor}[dref]{Corollary}

\newenvironment{proof}{\par\noindent{{\bf Proof.}}}{\hfill$\Box$
\medskip}

\title{Local analytic regularity in the linearized Calder\'on problem}
\author{Johannes Sj\"ostrand\footnote{ project NOSEVOL ANR
  2011 BS 010119 01}\\\small Institut de
  Math\'ematiques de Bourgogne,
  Universit\'e de Bourgogne\\
  \small 9 avenue Alain Savary - BP 47870\\
  \small 21078 Dijon cedex\\ \footnotesize
  johannes.sjostrand@u-bourgogne.fr\\
  \footnotesize and  UMR 5584 du CNRS\and Gunther Uhlmann\footnote{Partly supported by NSF and a Simons Fellowship}\\\small Department of Mathematics, University of Washington\\\small Box 354350\\ \small Seattle, WA, 98195\\
\footnotesize gunther@math.washington.edu}
 \date{}
\begin{document}
\maketitle
\begin{abstract}We consider the linearization of the
  Dirichlet-to-Neumann (DN) map as a function of the potential. We
  show that it is injective at a real analytic potential for
  measurements made at an open subset of analyticity of the
  boundary. More generally, we relate the analyticity up to the
  boundary of the variations of the potential to the analyticity of
  the symbols of the corresponding variations of the DN-map.

\medskip\centerline{\bf R\'esum\'e}
Nous consid\'erons la lin\'earis\'ee de
l'application de Dirichlet--Neumann (DN) comme fonction du potentiel
en un point donn\'e par un potentiel analytique.  Nous montrons
qu'elle est injective pour des m\'esures faites dans un ouvert o\`u le
bord est analytique. Plus g\'en\'eralement, nous lions l'analyticit\'e
jusqu'au bord des variations infinit\'esimales du potentiel \`a celle
des symboles des variations correspondantes le l'application DN.
\end{abstract}
\medskip 

\tableofcontents

\section{Introduction}\label{int}
\setcounter{equation}{0}

In this paper we consider the {\sl linearized} Calder\'on problem with local partial data and related problems. We review first briefly Calder\'on's problem including the case of partial data. For a more complete review see \cite{U}.

Calder\'on's problem is, roughly speaking, the question of whether one can determine the electrical
conductivity of a medium by making voltage and current measurements at the boundary of the medium. 
This inverse method is also called Electrical Impedance Tomography (EIT).
We describe the problem more precisely below.

Let $\Omega\subseteq\mathbb{R}^n$ be a bounded domain with smooth
boundary.
The electrical conductivity of $\Omega$
is represented by a bounded and positive function $\gamma(x)$. In
the absence of sinks or sources of current  the equation for the
potential is given by
\begin{equation}\label{eq:0.1}
\nabla\cdot(\gamma\nabla u)=0\hbox{ in }\Omega
\end{equation}
since, by Ohm's law, $\gamma\nabla u$ represents the current flux.
Given a potential $f\in H^{\frac{1}{2}}(\partial\Omega)$ on the
boundary the induced potential $u\in H^1(\Omega)$ solves the
Dirichlet problem
\begin{equation}\label{eq:0.2}
\begin{array}{rcl}
\nabla\cdot(\gamma\nabla u) & = & 0 \hbox{ in }\Omega, \\u\big|_{\partial\Omega} & = & f.
\end{array}
\end{equation}
The Dirichlet to Neumann (DN) map, or voltage to current map, is given by
\begin{equation}\label{eq:0.3}
\Lambda_\gamma(f)=\left(\gamma\frac{\partial u}{\partial \nu
}\right)\Big|_{\partial\Omega}
\end{equation}
where $\nu$ denotes the unit outer normal to $\partial\Omega$.
The inverse problem is to determine $\gamma$ knowing
$\Lambda_\gamma$. 

The local Calder\'on problem, or the Calder\'on problem with partial
data, is the question of whether one can determine the conductivity by
measuring the DN map on subsets of the boundary for voltages supported
in subsets of the boundary. In this paper we consider the case when
the support of the voltages and the induced current fluxes are
measured in the same open subset $\Gamma$.  More conditions on this
open set will be stated later.  If $\gamma\in
C^\infty(\overline\Omega)$ the DN map is a classical
pseudodifferential operator of order 1. It was shown in \cite{SU1}
that its full symbol computed in boundary normal coordinates near a
point of $\Gamma$ determines the Taylor series of $\gamma$ at the
point giving another proof of the result of Kohn and Vogelius
\cite{KV}. In particular this shows that real-analytic conductivities
can be determined by the local DN map. This result was generalized in
\cite{LU} to the case of anisotropic conductivities using a
factorization method related to the methods of this paper.  Interior
determination was shown in dimension $n\ge 3$ for $C^2$ conductivities
\cite{SU2}. This was extended to $C^1$ conductivities in \cite{HT}. In
two dimensions uniqueness was proven for $C^2$ conductivities in
\cite{N1} and for merely $L^\infty$ conductivities in \cite{AP}.  The
case of partial data in dimension $n\ge 3$ was considered in
\cite{BU}, \cite{KSU}, \cite{Is}, \cite{KS1}, \cite{IY}. The two
dimensional case was solved in \cite{IUY}. See \cite{KS2} for a
review. However it is not known at the present whether one can
determine uniquely the conductivity if one measures the DN map on an
arbitrarily open subset of the boundary applied to functions supported
in the same set. We refer to these type of measurements as the local
DN map.

The $\gamma\rightarrow \Lambda_\gamma$ is not linear. In this paper we
consider the linearization of the partial data problem at a
real-analytic conductivity for real-analytic $\Gamma$. We prove that
the linearized map is injective. In fact we prove a more general
statement, {see Theorem \ref{int3}.}

As in many works on Calder\'on's problem one can reduce the problem to
a similar one for the Schr\"odinger equation (see for instance
\cite{U}). This result uses that one can determine from the DN map the
conductivity and the normal derivative of the conductivity. This
result is only valid for the local DN map.  One can then consider the
more general problem of determining a potential from the corresponding
DN map.  The same is valid for the case of partial data and the
linearization.  It was shown in \cite{DKSU} that the linearization of
the local DN map at the $0$ potential is injective. We consider the
linearization of the local DN map at any real analytic potential
assuming that the local DN map is measured an an open real-analytic
set.  We now describe more precisely our results in this setting.

Consider the Schr\"odinger operator $P=\Delta -V$ on the open set $\Omega \Subset
{\bf R}^n$ where the boundary $\partial \Omega $ is smooth (and later
assumed to be analytic in the most interesting region). Assume that
$0$ is not in the spectrum of the Dirichlet realization of $P$. Let
$G$ and $K$ denote the corresponding Green and Poisson operators. Let
$\gamma :C^\infty (\overline{\Omega })\to C^\infty (\partial \Omega )$
be the restriction operator, $\nu $ the exterior normal. If
$x_0\in \partial \Omega $, we can choose local coordinates
$y=(y_1,...,y_n)$, centered at $x_0$ so that $\Omega $ is given by
$y_n>0$, $\partial _\nu
=-\partial _{y_n}$. If $\partial \Omega $ is analytic near $x_0$, we
can choose the coordinates to be analytic. 

\par The Dirichlet
to Neumann (DN) operator is 
\begin{equation}\label{int.1}
{\Lambda}={\gamma \partial _\nu K}.
\end{equation}

Consider a smooth deformation of smooth potentials, possibly
complex-valued,
\begin{equation}\label{int.2}\begin{split}
&\mathrm{neigh\,}(0,{\bf R})\ni t\mapsto P_t=\Delta -V_t,\\
&V_t(x)=V(t,x)\in C^\infty (\mathrm{neigh\,}(0,{\bf R})\times
\overline{\Omega };{\bf R}).
\end{split}
\end{equation}
Let $G_t$, $K_t$ be the Green and Poisson kernels for $P_t$, so that 
$$
\begin{pmatrix}P_t\\ \gamma \end{pmatrix}:\ C^\infty (\overline{\Omega
})\to C^\infty (\overline{\Omega
})\times C^\infty (\partial \Omega )
$$
has the inverse 
$$
\begin{pmatrix}G_t &K_t\end{pmatrix}.
$$
Then, denoting $t$-derivatives by dots,
\[
\begin{pmatrix}\dot{G}_t & \dot{K}_t \end{pmatrix}=
-\begin{pmatrix}G_t &K_t\end{pmatrix} \begin{pmatrix}\dot{P}_t\\
  0\end{pmatrix}
\begin{pmatrix}G_t &K_t\end{pmatrix}=-\begin{pmatrix}G_t\dot{P}_tG_t & G_t\dot{P}_tK_t\end{pmatrix}
\]
that is,
\begin{equation}\label{int.3}
\dot{G}=-G\dot{P}G,\quad \dot{K}=-G\dot{P}K,
\end{equation}
and consequently,
\begin{equation}\label{int.4}
\dot{{\Lambda}}=-\gamma \partial _\nu G\dot{P}K.
\end{equation}
Using the Green formula, we see that 
\begin{equation}\label{int.5}
\gamma \partial _\nu G=K^\mathrm{t},
\end{equation}
where $K^\mathrm{t}$ denotes the transposed operator.

In fact, write the Green formula,
$$
\int_\Omega ((Pu_1)u_2-u_1Pu_2)dx=\int_{\partial \Omega
}(\partial_\nu u_1 u_2-u_1\partial _\nu u_2)S(dx) 
$$
and put $u_1=Gv$, $u_2=Kw$ for $v\in C^\infty (\overline{\Omega })$,
$w\in C^\infty (\partial \Omega )$
$$
\int_\Omega vKwdx=\int_{\partial \Omega }(\gamma \partial _\nu Gv) wS(dx)
$$
and (\ref{int.5}) follows.

(\ref{int.4}) becomes
\begin{equation}\label{int.6}
\dot{{\Lambda}}=-K^\mathrm{t}\dot{P}K=K^\mathrm{t}\dot{V}K.
\end{equation}

\par The linearized Calder\'on problem is the following: If $V_t=V+tq$,
determine $q$ from $\dot{{\Lambda}}_{t=0}$. The corresponding partial
data problem is to recover $q$ or some information about $q$ from
local information about $\dot{{\Lambda}}_{t=0}$. From now on, we restrict the
attention to $t=0$. In this paper we shall study the following linearized
baby problem: Assume that $V$ and $\partial \Omega $ are analytic near
some point $x_0\in \partial \Omega $. We also assume that $V$ is smooth.  If $\dot{{\Lambda}}$ (for $t=0$) is an analytic
pseudodifferential operator near $x_0$, can we conclude that $q$ is
analytic near $x_0$?  Here,
\begin{equation}\label{int.7}
\dot{{\Lambda}}=K^\mathrm{t}qK
\end{equation}
and we shall view the right hand side as a Fourier integral operator
acting on $q$. 

Actually this problem is overdetermined in the sense that the symbol
of a pseudodifferential operator on the boundary is a function of
$2(n-1)$ variables, while $q$ is a function of $n$ variables and we
have $2(n-1)\ge n$ for $n\ge 2$ with equality precisely for $n=2$. In
order to have a non-overdetermined problem we shall only consider the
symbol $\sigma _{\dot{{\Lambda}}}(y',\eta ' )$ of $\dot{{\Lambda}}$ along
a half-ray in $\eta ' $, i.e. we look at $\sigma _{\dot{{\Lambda
    }}}(y',t \eta _0' )$ for some fixed $\eta _0'\ne 0$ and for some
local coordinates as above. Assuming this restricted symbol to be a
classical analytic symbol near $y'=0$ and the potential $V=V_0$ to be
analytic near $y=0$ (i.e. near $x_0$), we shall show that $q$ is
real-analytic up to the boundary near $x_0$ (corresponding to $y=0$).

To formulate the result more precisely, we first make some remarks
about the analytic singular support of the Schwartz kernels  of $K$
and $K^\mathrm{t}qK$, then we recall the notion of classical analytic
pseudo\-differential operators. Assume that $W\subset {\bf R}^n$ is an
open neighborhood of $x_0\in \partial {\cal O}$ and that 
\begin{equation}\label{int.8}\partial
  \Omega \hbox{ and } V \hbox{ are analytic in }W. \end{equation}
For simplicity we shall use the same symbol to denote operators and
their Schwartz kernels. 
Then we have
\begin{lemma}\label{int1} The Schwartz kernel
$K(x,y')$ is analytic with respect to $y'$, locally uniformly on the set
$$
\{(x,y')\in \overline{\Omega }\times (\partial \Omega \cap W);\, x\ne y' \}.
$$
\end{lemma}
\begin{proof}
Using (\ref{int.5}) we can write $K(x,y')=\gamma \partial _{\nu
}u(y')$, where $u=G(x,.)$ solves the Dirichlet problem 
$$
(\Delta -V) u=\delta (\cdot -x),\ \gamma u=0,
$$
and from analytic regularity for elliptic boundary value problems, we
get the lemma. (When $x\in \partial \Omega $, we view $G(x,y)$ away
from $y=x$ as the limit of $G(x_j,y)$ when $\Omega \ni x_j\to x$.)
\end{proof}

We next define the notion of symbol up to exponentially small
contributions. For that purpose we assume that $X$ is an analytic
manifold and consider an operator
\begin{equation}\label{int.10}A:\, C_0^\infty (X)\to C^\infty
  (X)\end{equation}
which is also continuous
\begin{equation}\label{int.11} {\cal E}'(X)\to {\cal
    D}'(X).\end{equation}
Assume (as we have verified for $K^\mathrm{t}qK$ with $n$ replaced by
$n-1$ and with $X=\partial \Omega \cap W$) that the distribution
kernel $A(x,y)$ is analytic away from the diagonal. After restriction
to a local analytic coordinate chart, we may assume that $X\subset
{\bf R}^n$ is an open set. The symbol of $A$ is formally given on $T^*X$ by
$$
\sigma _A(x,\xi )=e^{-ix\cdot \xi }A(e^{i(\cdot )\cdot \xi })=\int e^{-i(x-y)\cdot \xi }A(x,y)dy.
$$
In the usual case of $C^\infty $-theory, we give a meaning to this
symbol up to ${\cal O}(\langle \xi \rangle ^{-\infty })$ by
introducing a cutoff $\chi (x,y)\in C^\infty (X\times X)$ which is
properly supported and equal to 1 near the diagonal. In the analytic
category we would like to have an exponentially small undeterminacy,
and the use of special cut-offs becoming more complicated, we prefer
to make a contour deformation.

\par For $x$ in a compact subset of $X$, let $r>0$ be small enough and
define for $\xi \ne 0$,
\begin{equation}\label{int.12}
\sigma _A(x,\xi )=\int_{x+\Gamma _{r,\xi }}e^{i(y-x)\cdot \xi }A(x,y)dy,
\end{equation}
where 
$$
\Gamma _{r,\xi }:B(0,r)\ni t\mapsto t+i\chi
\left(\frac{t}{r}\right)r\frac{\xi }{|\xi |}\in {\bf C}^n
$$
and $\chi \in C^\infty (B(0,1);[0,1])$ is a radial function which
vanishes on $B(0,1/2)$ and is equal to 1 near $\partial B(0,1)$. Thus
the contour $x+\Gamma _{r,\xi } $ coincides with ${\bf R}^n$ near
$y=x$ and becomes complex for $t$ close to the boundary of $B(0,r)$.
Along this contour,
$$
|e^{i(y-x)\cdot \xi }|=e^{-\chi (t/r)r|\xi |}
$$
is bounded by 1 and for $t$ close to $\partial B(0,r)$ it is
exponentially decaying in $|\xi|$. Thus from Stokes' formula it is
clear that $\sigma _A(x,\xi )$ will change only by an exponentially
small term if we modify $r$. More generally, for $(x,\xi )$ in a conic
neighborhood of a fixed point $(x_0,\xi _0)\in X\times S^{n-1}$ we change
$\sigma _A(x,\xi )$ only by an exponentially small term if we replace
the contour in (\ref{int.12}) by $x_0+\Gamma _{r,\xi _0}$ and we then
get a function which has a holomorphic extension to a conic
neighborhood of $(x_0,\xi _0)$ in ${\bf C}^n\times ({\bf C}^n\setminus
\{0 \})$.  
\begin{remark}\label{int2.7}
Instead of using contour deformation  to define $\sigma _A$, we can
use an almost analytic cut-off in the following way. Choose $C>0$ so
that
$$
1=\int Ch^{-\frac{n}{2}}e^{-\frac{(y-t)^2}{2h}}dt,
$$
and put 
$$
e_t(y)=\widetilde{\chi }(y-t)Ch^{-\frac{n}{2}}e^{-\frac{(y-t)^2}{2h}},$$
where $\widetilde{\chi }\in C_0^\infty ({\bf R}^n)$ is equal to 1 near
$0$ and has its support in a small neighborhood of that point. Then if
$\chi $ is another cut-off of the same type, we see by
contour deformation that 
$$\sigma _A(x,\xi )=e^{-ix\cdot \xi }A(\int \chi
(t-x)e_te^{i(\cdot )\cdot \xi }dt)$$
up to an exponentially decreasing term.
\end{remark}

In the following definition we recall the notion of analytic symbols
in the sense of L.~Boutet de Monvel and P.~Kr\'ee \cite{BoKr67}. We
will avoid to use the corresponding notion of analytic
pseudodifferential operator, since it involves some special facts
about the distribution kernel of the operator that will not be needed.
\begin{dref}\label{int2.8}  We say that $\sigma _A$ is a classical
analytic symbol (cl.a.s.) of order $m$ on $X\times {\bf R}^n$ if the
following holds:

\par There exist holomorphic functions $p_{m-j}(x,\xi )$ on a fixed complex
conic neighborhood $V$ of $X\times \dot{{\bf R}}^n$ such that
\begin{equation}\label{int.13}
p_k(x,\xi ) \hbox{ is positively homogeneous of degree }k\hbox{ in
}\xi ,
\end{equation}
\begin{equation}\label{int.14}\begin{split}
&\forall K\Subset V\cap \{(x,\xi );\, |\xi |=1 \},\ \exists \,
C=C_K\hbox{ such that}\\
&|p_{m-j}(x,\xi )|\le C^{j+1}j^j,\hbox{ on }K.
\end{split}
\end{equation}
\begin{equation}\label{int.15}\begin{split}
&\forall K\Subset X,\hbox{ and every }C_1>0,\hbox{ large enough,
}\exists C_2>0,\\
&\hbox{such that }|\sigma _A(x,\xi )-\sum_{0\le j\le |\xi |/C_1}p_{m-j}(x,\xi )|\le C_2
e^{-|\xi |/C_2},\\ &(x,\xi )\in K\times {\bf R}^n,\ |\xi |\ge 1.
\end{split}
\end{equation}
The formal sum $\sum_0^\infty p_{m-j}(x,\xi )$ is called a formal
cl.a.s. when (\ref{int.13}), (\ref{int.14}) hold. We define
cl.a.s. and formal cl.a.s. on open conic subsets of $X\times \dot{{\bf
  R}}^n$ and on other similar sets by the obvious modifications of the
above definitions. If $p(x,\xi )$ is a cl.a.s. on $X\times {\bf R}^n$
and if $\xi _0\in \dot{{\bf R}}^n$, then
$$
q(x,\tau ):=p(x,\tau \xi _0)
$$
is a cl.a.s. on $X\times {\bf R}_+$. \end{dref}

\par The main result of this work is
\begin{theo}\label{int3}
  Let $x_0\in \partial \Omega $ and assume that $\partial \Omega $ and
  $V$ are analytic near that point. Let $q\in L^\infty(\Omega )$. Choose local analytic coordinates $y'=(y_1,...,y_{n-1})$
  on $\mathrm{neigh\,}(x_0,\partial \Omega )$, centered at $x_0$.

\par If { $\dot{\Lambda }(x',y')$ is analytic in $W'\times W'\setminus
\mathrm{diag\,}(W'\times W')$, where $W'$ is a neighborhood of $0$ in
${\bf R}^n$  and} $\sigma _{\dot{{\Lambda}}}(y',\tau \eta _0')$ is a cl.a.s. on
$\mathrm{neigh\,}(0,{\bf R}^{n-1})\times {\bf R}_+$, then $q$ is
analytic up to the boundary in a neighborhood of
$x_0$. {Here $\dot{\Lambda }(x',y')$ denotes the
  distribution kernel of $\dot{\Lambda }$.}
\end{theo}

\par We have a simpler direct result.
\begin{prop}\label{int4}
Let $x_0$, $\partial \Omega $, $V$ be as in Theorem \ref{int3} and
choose analytic coordinates as there. If $q\in L^\infty (\Omega )$ is
analytic up to the boundary near $x_0$, then { there
  exists a neighborhood $W'$ of $0$ in ${\bf R}^{n-1}$ such that
  $\dot{\Lambda }(x',y')$ is analytic in $W'\times W'\setminus
  \mathrm{diag\,}(W'\times W')$ and} $\sigma
_{\dot{{\Lambda}}}$ is a cl.a.s. near $y'=0$. 
\end{prop}

We get the following immediate consequence
\begin{cor}\label{int5}
Under the conditions of the previous theorem the map
$$q \rightarrow  \dot{{\Lambda}}$$ is injective.
\end{cor}

This follows from the previous result since $q$ must be analytic on
$W$ and the Taylor series of $q$ vanishes on $W$ then $q=0$ on the set
where $q$ is analytic.

{
\begin{remark}\label{int6}
As we have seen, we need the analyticity of $\dot{\Lambda }(x',y')$
away from the diagonal in order to define $\sigma _{\dot{\Lambda
  }}(y',\eta ')$ up to an exponentially small term. However, the proof
of the theorem only uses the estimates (\ref{int.14}) for the formal
asymptotic expansion of the symbol and the fact that an FBI-transform
of $\sigma _{\dot{\Lambda }}(y',\eta ')$ with respect to $y'$ is an
analytic symbol, and these objects can be defined without that
analyticity assumption. 
\end{remark}}

Most of the paper will be devoted to the proof of Theorem \ref{int3},
and in Section \ref{prp} we will prove Proposition \ref{int4}.

\section{Heuristics and some remarks about the Laplace transform}\label{la}
\setcounter{equation}{0}

Let us first explain heuristically why some kind of Laplace transform
will appear. Assume that $x_0\in \partial \Omega $ and that $V$ and
$\partial \Omega $ are analytic near that point. Choose local analytic
coordinates $$y=(y_1,...,y_{n-1},y_n)=(y',y_n)$$ centered at $x_0$ such
that the set
$\Omega $ coincides near $x_0$ (i.e. $y=0$) with the half-space ${\bf R}^n_+=\{y\in {\bf
  R}^n;\, y_n>0 \}$. Assume also (for this heuristic discussion) that
we know that $q(y)=q(y',y_n)$ is analytic in $y'$ and that the
original Laplace operator remains the standard Laplace operator also
in the $y$-coordinates. Then up to a smoothing operator, the Poisson operator
is of the form
$$Ku(y)=\frac{1}{(2\pi )^{n-1}}\int e^{i(y'-w')\cdot \eta '-y_n|\eta
  '|}a(y,\eta ')u(w')dw'd\eta ',$$
where the symbol $a$ is equal to 1 to leading order. We can view
$K$, $q$, $K^\mathrm{t}$ as pseudodifferential operators in $y'$ with
operator valued symbols. $K$ has the operator valued symbol
\begin{equation}\label{la.1}
K(y',\eta '):{\bf C}\ni z\mapsto ze^{-y_n|\eta '|}a(y,\eta ')\in
L^2([0,+\infty [_{y_n}).
\end{equation}
The symbol of multiplication with $q$ is independent of $\eta
'$ and equals multiplication with $q(y',\cdot )$. The symbol of
$K^\mathrm{t}$ is to leading order
\begin{equation}\label{la.2}
K^\mathrm{t}(y',\eta '):\ L^2([0,+\infty [_{y_n})\ni f(y_n)\mapsto
\int_0^\infty e^{-y_n|\eta '|}a(y,-\eta ')f(y_n)dy_n\in {\bf C}.
\end{equation}
For simplicity we set $a=1$ in the following discussion.
To leading order the symbol of $\dot{{\Lambda}}$ is 
\begin{equation}\label{la.3}
\sigma _{\dot{{\Lambda}}}(y',\eta ')=\int_0^\infty e^{-2y_n |\eta '|}
q(y',y_n)dy_n=({\cal L}q(y',\cdot ))(2|\eta '|),
\end{equation}
where
$${\cal L}f(\tau )=\int_0^\infty e^{-t\tau } f(t)dt$$
is the Laplace transform. 

\par Now we fix $\eta _0'\in \dot{{\bf R}}^{n-1}$ and assume that $\sigma _{\dot{{\Lambda}}}(y',\tau \eta _0')$ is a
cl.a.s. on $\mathrm{neigh\,}(0,{\bf R}^{n-1})\times {\bf R}_+$,
\begin{equation}\label{la.4}
\sigma _{\dot{{\Lambda}}}(y',\tau \eta_0 ')\sim \sum_1^\infty n_k(y',\tau ),
\end{equation}
where $n_k$ is analytic in $y'$ in a fixed complex neighborhood of
$0$, (positively) homogeneous of degree $-k$ in $\tau $ and satisfying
\begin{equation}\label{la.5}
|n_k(y',\tau )|\le C^{k+1}k^k|\tau |^{-k}.
\end{equation}
More precisely for $C>0$ large enough, there exists $\widetilde{C}>0$
such that
\begin{equation}\label{la.6}
  |\sigma _{\dot{{\Lambda}}}(y',\tau \eta_0 ')-\sum_1^{[|\eta
    '|/C]}n_k(y',\tau )|\le \widetilde{C} \exp (-\tau /\widetilde{C})
\end{equation}
on the real domain.

From (\ref{la.3}) we also have 
\begin{equation}\label{la.7}
  |({\cal L}q(y',\cdot ))(2|\eta _0'|\tau )-\sum_1^{[|\eta
    '|/C]}n_k(y',\tau)|\le \exp (-\tau /\widetilde{C}),
\end{equation} 
for $y'\in \mathrm{neigh\,}(0,{\bf R}^{n-1})$, $\tau \ge 1$. In this
heuristic discussion we assume that (\ref{la.7}) extends to $y'\in
\mathrm{neigh\,}(0,{\bf C}^{n-1})$.  It then follows that $q(y',y_n)$
is analytic for $y_n$ in a neighborhood of $0$, from the following
certainly classic result about Borel transforms.
\begin{prop}\label{la1}
Let $q\in L^\infty ([0,1])$ and assume that for some $C,\widetilde{C}>0$, 
\begin{equation}\label{la.8}
|{\cal L}q(\tau )-\sum_0^{[\tau /C]}q_k\tau ^{-(k+1)}|\le e^{-\tau
  /\widetilde{C}},\ \tau >0,
\end{equation}
\begin{equation}\label{la.9}
|q_k|\le \widetilde{C}^{k+1}k^k.
\end{equation}
Then $q$ is analytic in a neighborhood of $t=0$. The converse also
holds.
\end{prop}
\begin{proof}
We shall first show the converse statement, namely that if $q$ is
analytic near $t=0$, then (\ref{la.8}), (\ref{la.9}) hold. We start by
computing the Laplace transform of powers of $t$.

\par For $ \tau >0$, $a>0$, $k\in {\bf N}$, we have
\begin{equation}\label{la.10}
\int_0^\infty e^{-t\tau }t^kdt=\frac{k!}{\tau ^{k+1}}.
\end{equation}
In fact, the integral to the left is equal to 
$$
(-\partial _\tau )^k\left(\int_0^\infty e^{-t\tau } dt \right)=
(-\partial _\tau )^k\left(\frac{1}{\tau } \right).
$$

\par Next, for $a>0$, we look at 
\begin{equation}\label{la.11}\begin{split}
&\frac{1}{k!}\int_0^a e^{-t\tau }t^k dt=\frac{1}{\tau
  ^{k+1}}\left(1-\frac{\tau ^{k+1}}{k!}\int_a^\infty  e^{-t\tau }t^k
  dt \right)\\
&=\frac{1}{\tau ^{k+1}}\left(1-\int_{a\tau }^\infty e^{-s}\frac{s^k}{k!}ds \right).
\end{split}
\end{equation}
Let first $\tau \in ]0,\infty [$ be large. 
For $0<\theta <1$, we write for $s\ge 0$,
$$
\frac{s^k}{k!}e^{-s}=\theta ^{-k}\underbrace{\frac{(\theta s)^k}{k!}e^{-\theta
  s}}_{\le 1}e^{-(1-\theta )s}\le \theta ^{-k}e^{-(1-\theta ) s}.
$$
Thus,
\begin{equation}\label{la.12}
\int_{a\tau }^\infty e^{-s}\frac{s^k}{k!}ds\le\theta ^{-k}\int_{a\tau
}^\infty e^{-(1-\theta )s}ds = \frac{\theta ^{-k}e^{-(1-\theta )a\tau
  }}{1-\theta }.
\end{equation}
We will estimate this for $k\le a\tau /{\cal O}(1)$. Under the
apriori assumption that $\theta \le 1-\frac{1}{{\cal O}(1)}$, we look
for $\theta $ that minimizes the enumerator 
$$\theta ^{-k}e^{-(1-\theta
  )a\tau }=e^{-[(1-\theta )a\tau +k\ln \theta ]}.$$ 
Setting the
derivative of the exponent equal to zero, we are led to the choice
$\theta =\frac{k}{a\tau }$. Assume that
\begin{equation}\label{la.13}
\frac{k}{a\tau }\le \theta _0<1.
\end{equation}
Then,
\[
(1-\theta )a\tau +k\ln \theta =a\tau \left( 1-\frac{k}{a\tau
}+\frac{k}{a\tau }\ln \frac{k}{a\tau }\right)=a\tau
(1-f(\frac{k}{a\tau })),
\]
where 
$$
f(x)=x+x\ln\frac{1}{x},\ 0\le x\le 1.
$$
Clearly $f(0)=0$, $f(1)=1$ and for $0<x<1$ we have $f'(x)=\ln
\frac{1}{x}>0$, so $f$ is strictly increasing on $[0,1]$. In view of
(\ref{la.13}) we have
$$
(1-\theta )a\tau +k\ln \theta \ge a\tau (1-f(\theta _0)),
$$
and (\ref{la.12}) gives
\begin{equation}\label{la.14}
\int_{a\tau }^\infty e^{-s}\frac{s^k}{k!}ds\le \frac{e^{-a\tau
    (1-f(\theta _0))}}{1-\theta _0}.
\end{equation}
Using this in
(\ref{la.11}), we get
\begin{equation}\label{la.15}\begin{split}
&\frac{1}{k!}\int_0^a e^{-t\tau }t^k dt=\frac{1}{\tau ^{k+1}}(1+{\cal
  O}(1)e^{- a\tau /C(\theta _0)}),\\
&\hbox{for }\frac{k}{a\tau }\le \theta _0<1,\hbox{ where }C(\theta _0)>0.
\end{split}
\end{equation}

\par Now, assume that $q\in C([0,1])$ is analytic near $t=0$. Then for
$t\in [0,2a]$, $0<a\ll 1$, we have
$$
q(t)=\sum_0^\infty \frac{q^{(k)}(0)}{k!}t^k,
$$
where
\begin{equation}\label{la.16}
\frac{|q^{(k)}(0)|}{k!}\le \widetilde{C}\frac{1}{(2a)^k},
\end{equation}
so
$$
|q(t)-\sum_0^{[\tau /C]}\frac{q^{(k)}(0)}{k!}t^k|\le \widetilde{C} e^{-\tau
  /\widetilde{C}},\ 0\le t\le a.
$$
Hence,
$$
{\cal L}q=\sum_0^{[\tau /C]}\frac{q^{(k)}(0)}{\tau ^{k+1}}+{\cal
  O}(e^{-\tau /\widetilde{C}})+\underbrace{{\cal L}(1_{[a,1]}q)(\tau
  )}_{={\cal O}(e^{-\tau /\widetilde{C}})}
$$
and we obtain (\ref{la.8}) with $q_k=q^{(k)}(0)$, while (\ref{la.9})
follows from (\ref{la.16}).

\par We now prove the direct statement in the proposition, so we take
$q\in L^\infty ([0,1])$ satisfying (\ref{la.8}), (\ref{la.9}). Put for
$a>0$ small,
$$
\widetilde{q}(t)=q(t)-1_{[0,a]}(t)\sum_0^\infty \frac{q_k}{k!}t^k.
$$
The proof of the converse part shows that
\begin{equation}\label{la.17}
|{\cal L}\widetilde{q}(\tau )|\le e^{- \tau /\widetilde{C}},
\end{equation}
where $\widetilde{C}$ is a new positive constant, and it suffices to
show that
\begin{equation}\label{la.18}
\widetilde{q}\hbox{ vanishes in a neighborhood of }0.
\end{equation}

\par We notice that ${\cal L}\widetilde{q}$ is a bounded holomorphic
function in the right half-plane. We can therefore apply the
Phragm\'en-Lindel\"of theorem in each sector $\mathrm{arg\,}\tau \in
[0 ,\frac{\pi }{2}]$ and $\mathrm{arg\,}\tau \in
[-\frac{\pi }{2},0 ]$ to the holomorphic function
$$
e^{\tau /\widetilde{C}}{\cal L}\widetilde{q}(\tau )
$$
and conclude that this function is bounded in the right half-plane:
\begin{equation}\label{la.19}
|{\cal L}\widetilde{q}(\tau )|\le {\cal O}(1)e^{-\Re \tau
  /\widetilde{C}},\ \Re \tau \ge 0.
\end{equation}
Now, ${\cal L}\widetilde{q}(i\sigma )={\cal F}\widetilde{q}(\sigma )$,
where ${\cal F}$ denotes the Fourier transform, and Paley-Wiener's
theorem allows us to conclude that
$\mathrm{supp\,}\widetilde{q}\subset [\frac{1}{\widetilde{C}},1]$.
\end{proof}
\section{The Fourier integral operator $q\mapsto \sigma _{\dot{{\Lambda}}}$}\label{fop}
\setcounter{equation}{0} Assume that $\partial \Omega $ and $V$ are
analytic near the boundary point $x_0$. Let $y'=(y_1,...,y_{n-1})$ be
local analytic coordinates on $\partial \Omega $, centered at
$x_0$. Then we can extend $y'$ to analytic coordinates
$y=(y_1,...,y_{n-1},y_n)=(y',y_n)$ in a full neighborhood of $x_0$,
where $y'$ are extensions of the given coordinates on the boundary and
such that $\Omega $ is given (near $x_0$) by $y_n>0$ and
\begin{equation}\label{fop.1}
-P=D_{y_n}^2+R(y,D_{y'}),
\end{equation}
where $R$ is a second order elliptic differential operator in $y'$
with positive principal symbol $r(y,\eta ')$. (Here we neglect a
contribution $f(y)\partial _{y_n}$ which can be eliminated by
conjugation.) Then there is a neighborhood $W\subset {\bf R}^n$ of
$y=0$ and a cl.a.s. $a(y,\xi ')$ on $W\times {\bf R}^{n-1}$ of order
$0$ such that
\begin{equation}\label{fop.2}
Ku(y)=\frac{1}{(2\pi )^{n-1}}\iint e^{i(\phi (y,\xi ')-\widetilde{y}'\cdot \xi
  ')}a(y,\xi ')u(\widetilde{y}')d\widetilde{y}'d\xi '+K_au(y),
\end{equation}
for $y\in W$, $u\in C_0^\infty (W\cap \partial \Omega )$. The
distribution kernel of $K_a$ is analytic on $W\times (W\cap \partial
\Omega )$ and we choose a realization of $a$ which is analytic in
$y$. 
$\phi $ is the solution of the Hamilton-Jacobi problem
\begin{equation}\label{fop.3}\begin{split}
    &(\partial _{y_n}\phi )^2+r(y,\phi '_{y'})=0,\ \Im \partial
    _{y_n}\phi >0,\\ &\phi (y',0,\xi ')=y'\cdot \xi '.
\end{split}
\end{equation}
This means that we choose $\phi $ to be the solution of 
\begin{equation}\label{fop.4}
\partial _{y_n}\phi -ir(y,\phi '_{y'})^{1/2}=0,
\end{equation}
with the natural branch of $r^{1/2}$ with a cut along the real
negative axis. 

\par To see this, recall (by the analytic WKB-method, cf. \cite{Sj82},
Ch.~9) that we can construct the first term
$K_\mathrm{fop}u$ in the right hand side of (\ref{fop.2}) such that
$PK_\mathrm{fop}$ has analytic distribution kernel and $\gamma
K_\mathrm{fop}=1$. It then follows from local analytic regularity in
elliptic boundary value problems, that the remainder operator $K_a$
has analytic distribution kernel.

\par We notice that 
\begin{equation}\label{fop.4.5}
  K(e^{ix'\cdot \xi '})=e^{i\phi (y,\xi ')}a(y,\xi ')+{\cal O}(e^{-|\xi '|/C}), \end{equation}
since the first term to the right solves the problem 
$$
Pu=0,\quad {{u}_\vert}_{y_n=0}=e^{iy'\cdot \xi '}
$$
with an exponentially small error in the first equation. When $V$ is real, then $K$ is real, so $K(e^{ix'\cdot (-\xi
  ')})=\overline{K(e^{ix'\cdot \xi '})}$. It follows that
\begin{equation}\label{fop.5}
\phi (y,-\xi ')=-\overline{\phi (y,\xi ')},\ a(y,-\xi
')=\overline{a(y,\xi ')}
\end{equation}
without any error in the last equation when viewing $a$ as a formal
cl.a.s. Now $\phi $ and the leading homogeneous term $a_0$ in $a$ are
independent of $V$, so if we drop the reality assumption on $V$, the
first part of (\ref{fop.5}) remains valid and the second part is valid
to leading order. 

\par We shall now view $\dot{{\Lambda}}=K^\mathrm{t}qK=K^*qK$ as a
pseudodifferential operator in the classical quantization. In this
section we proceed formally in order to study the associated geometry. A
more efficient analytic description will be given later for the left
composition with an FBI-transform in $x'$. The symbol
becomes 
\[
\begin{split}
&\sigma _{\dot{{\Lambda}}}(x',\xi ')=e^{-ix'\cdot \xi '}{\dot{\Lambda }}(e^{i(\cdot )\cdot \xi '})\\
&=(2\pi )^{1-n}\iint e^{i(x'\cdot (\eta '-\xi ')-\phi^*(y,\eta ')+\phi
  (y,\xi '))}a(y,-\eta ')a(y,\xi')q(y)dyd\eta ',
\end{split}
\]
where in general we write $f^*(z)=\overline{f(\overline{z})}$ for the
holomorphic extension of the complex conjugate of a function $f$. Here
we use (cf.~(\ref{fop.5})) that
$$
K^\mathrm{t}v(x')=(2\pi )^{1-n}\iint e^{i(x'\cdot \eta '-\phi
  ^*(y,\eta '))}a(y,-\eta ')v(y)dyd\eta '.
$$
\par Actually, rather than letting $\xi '$ tend to $\infty $ we
replace $\xi '$ with $\xi '/h$ where the new $\xi '$ is of length
$\asymp 1$ and $h\to 0$. This amounts to viewing $\dot{\Lambda }$ as a
semi-classical pseudodifferential operator with semi-classical symbol
$\sigma _{\dot{{\Lambda}}}(x',\xi ';h)=\sigma
_{\dot{{\Lambda}}}(x',\xi '/h)$. Thus,
\[
\begin{split}
&\sigma _{\dot{{\Lambda}}}(x',\xi ';h)=e^{-ix'\cdot \xi
  '/h}\dot{\Lambda }(e^{i(\cdot )\cdot \xi '/h})\\
&=(2\pi h)^{1-n}\iint e^{\frac{i}{h}(x'\cdot (\eta '-\xi ')-\phi^*(y,\eta ')+\phi
  (y,\xi '))}a(y,-\eta ';h)a(y,\xi';h)q(y)dyd\eta ',
\end{split}
\]
where $a(y,\xi ';h)=a(y,\xi '/h)$.

We have 
\begin{equation}\label{fop.6}
\phi (y,\xi ')=y'\cdot \xi '+\psi (y,\xi '),\ \phi^* (y,\eta ')=y'\cdot \eta '+\psi^* (y,\eta '), 
\end{equation}
where
\begin{equation}\label{fop.7}
\Im \psi \asymp y_n,\ \Re \psi ={\cal O}(y_n^2), 
\end{equation}
uniformly on every compact set which does not intersect the zero
section. (\ref{fop.5}) tells us that $\Re \psi $ is odd and $\Im \psi $ is
even with respect to the fiber variables $\xi '$ (and also positively
homogeneous of degree 1 of course). Using (\ref{fop.6}) in the formula for the
symbol of $\dot{{\Lambda}}$, we get
\begin{equation}\label{fop.8}
\begin{split}
&\sigma _{\dot{{\Lambda}}}(x',\xi ';h)\\
&=(2\pi h)^{1-n}\iint e^{\frac{i}{h}\Phi_M (x',\xi ',y,\eta ')}a(y,-\eta ';h)a(y,\xi';h)q(y)dyd\eta '\\
&=:Mq(x',\xi ';h),
\end{split}
\end{equation}
where
\begin{equation}\label{fop.9}
\Phi_M (x',\xi ',y,\eta ')=(x'-y')\cdot (\eta '-\xi ')+\psi (y,\xi
')-\psi ^*(y,\eta '),
\end{equation}
and $\eta '$ are the fiber variables. We shall see that this is a
nondegenerate phase function in the sense of L.~H\"ormander
\cite{Ho71} except for the fact that $\Phi_M $ is not homogeneous in
$\eta '$ alone, so $q\mapsto M_hq(x',\xi '):=Mq(x',\xi ';h)$ is a
semi-classical Fourier integral operator, at least formally. 

\par We fix a vector $\xi _0'\in \dot{{\bf R}}^{n-1}$ and consider
$\Phi_M $ in a neighborhood of $(x',y,\xi ',\eta ' )=(0,0,\xi
_0',\xi _0')\in {\bf
  C}^{4(n-1)+1}={\bf C}^{4n-3}$.
The critical set $C_{\Phi_M} $ of the phase ${\Phi_M} $ is given by $\partial _{\eta
  '}{\Phi_M} =0$, that is $x'-y'-\partial _{\eta '}\psi ^*(y,\eta ')=0$
or equivalently,
\begin{equation}\label{fop.11}
x'=y'+\partial _{\eta '}\psi ^*(y,\eta ').
\end{equation}
This is a smooth submanifold of codimension $n-1$ in ${\bf C}^{4n-3}$,
parametrized by $(y,\eta ',\xi ')\in \mathrm{neigh\,}((0,\xi _0',\xi
_0'),{\bf C}^{3n-2})$. We also see that ${\Phi_M} $ is a nondegenerate
phase function in the sense that $d\partial _{\eta '_1}{\Phi_M}
,...,d\partial _{\eta '_{n-1}}{\Phi_M}$ are linearly independent on
$C_{\Phi_M} $. Using the above parametrization, we express the graph of
the corresponding canonical relation $\kappa :{\bf C}^{2n}_{y,y^*}\to
{\bf C}^{4(n-1)}_{x',\xi ',{x'}^*,{\xi '}^*}$ (where we notice that
$4(n-1)\ge 2n$ with equality for $n=2$ and strict inequality for $n\ge
3$):
\begin{equation}\label{fop.12}
\begin{split}
&\mathrm{graph\,}(\kappa )=\{(x',\xi ',\partial _{x'}{\Phi_M} ,\partial
_{\xi '}{\Phi_M}; y,-\partial _y{\Phi_M} );\ (x',\xi',y,\eta ')\in C_{{\Phi_M} }
\} \\
&=\{ (y'+\partial _{\eta '}\psi ^*(y,\eta '),\xi ',\eta '-\xi
',\partial _{\xi '}\psi (y,\xi ')-\partial _{\eta '}\psi ^*(y,\eta
');\\ & y,-\partial _{y'}\psi (y,\xi ')+\partial _{y'}\psi ^*(y,\eta
')+\eta '-\xi ',-\partial _{y_n}\psi (y,\xi ')+\partial _{y_n}\psi
^*(y,\eta ')) \}
\end{split}
\end{equation}

The restriction to $y_n=0$ of this graph is the set of
points
\begin{equation}\label{fop.13}
(y',\xi ',\eta '-\xi ',0;y',0,\eta '-\xi ',-\partial _{y_n}\psi
(y',0,\xi')+\partial _{y_n}\psi ^* (y',0,\eta ')).
\end{equation}
It contains the point
\begin{equation}\label{fop.13.5}
(0,\xi _0',0,0;0,0,-2 \partial _{y_n}\psi (0,\xi _0'))=
(0,\xi _0',0,0;0,0,-2 ir(0,\xi _0')^{\frac{1}{2}})
\end{equation}

The tangent space at a point where $y_n=0$ is given by
\begin{equation}\label{fop.14}
\begin{split}
\{
(&\delta _{y'}+{\psi ^*}''_{\eta ',y_n}\delta _{y_n},\delta _{\xi
  '},\delta _{\eta '}-\delta _{\xi '},(\psi ''_{\xi ',y_n}(y,\xi ')-
{\psi ^*}''_{\eta ',y_n}(y,\eta '))\delta _{y_n};\\
&\delta _y,(-\psi ''_{y',y_n}(y,\xi ')+{\psi ^*}''_{y',y_n}(y,\eta
'))\delta _{y_n}+\delta _{\eta '}-\delta _{\xi '},\\ &(-\psi
''_{y_n,y}(y,\xi ')+{\psi ^*}''_{y_n,y}(y,\eta '))\delta _y+(-\psi
''_{y_n,\xi '}\delta _{\xi '}+{\psi ^*}''_{y_n,\eta '}\delta _{\eta '}))
  \}
\end{split}
\end{equation}
From (\ref{fop.14}) we see that at every point of
$\mathrm{graph\,}\kappa $ with $y_n=0$ and with $\eta '\approx \xi '$,
\begin{itemize}
\item[1)] The projection $\mathrm{graph\,}(\kappa )\to {\bf
    C}^{2n}_{y,y^*}$ has surjective differential,
\item[2)] The projection $\mathrm{graph\,}(\kappa )\to {\bf
    C}^{4(n-1)}_{x',\xi ',{x'}^*,{\xi '}^*}$ has injective differential.
\end{itemize}
In fact, since $\kappa $ is a canonical relation, 1) and 2) are
pointwise equivalent, so it suffices to verify 2). In other words, we
have to show that if
\begin{equation}\label{fop.15}
\begin{split}
0&=\delta _{y'}+{\psi ^*}''_{\eta ',y_n}\delta _{y_n},
\\  0&=\delta _{\xi
  '},\\
0&=\delta _{\eta '}-\delta _{\xi '},\\
0&=(\psi ''_{\xi ',y_n}(y,\xi ')-
{\psi ^*}''_{\eta ',y_n}(y,\eta '))\delta _{y_n},
\end{split}
\end{equation}
then $\delta _{y'}=0$, $\delta _{y_n}=0$, $\delta _{\xi '}=0$, $\delta
_{\eta '}=0$.

\par
When $y_n=0$ we have $\partial _{y_n}\psi ^*=-\partial _{y_n}\psi $ and when in addition $\eta
'\approx \xi '$ we see that the $(n-1)\times 1$ matrix in the 4:th
equation is non-vanishing, so this equation implies that $\delta
_{y_n}=0$. Then the first equation gives $\delta _{y'}=0$ and from the
2:nd and the 3:d equations we get $\delta _{\xi '}=0$ and $\delta
_{\eta '}=0$ and we have verified 2).

\par As an exercise, let us determine the image under $\kappa $ of the
complexified conormal bundle of the boundary, given by $y_n=0$,
${y^*}'=0$. From (\ref{fop.13}) we see that it is equal to the set of
all points
\begin{equation}\label{fop.16}
(x',\xi ',0,0) .
\end{equation} 
The subset of real points in (\ref{fop.16}) is the image
of the set of points $(y',0,0,y_n^*)$ such that $y'$ is real and
$y_n^*\in -i{\bf R}_+$.

\par Now restrict $(x',\xi ')$ to the set of $(x',t\eta _0')$, $x'\in
{\bf C}^{n-1}$, $t\in {\bf C}$, where $0\ne \eta _0'\in {\bf
  R}^{n-1}$. This means that we restrict the symbol of $\dot{\Lambda
}$ to the radial direction $\xi '\in {\bf C}\eta _0'$ and consider
\begin{equation}\label{fop.17}\begin{split}
&\sigma _{\dot{{\Lambda}}}(x',t\eta _0';h)=Mq(x',t\eta
_0';h)=:M_\mathrm{new}q(x',t;h)\\
&=(2\pi h)^{1-n}\iint e^{i\Phi _{M_\mathrm{new}}(x',t,y,\eta ')/h}a(y,-\eta ';h)a(y,\xi';h)q(y)dyd\eta ',
\end{split}
\end{equation}
where 
\begin{equation}\label{fop.18}\begin{split}
   & \Phi _{M_\mathrm{new}}(x',t,y;\eta ')=\Phi_M (x',t\eta _0',y;\eta ')\\
   &= \psi (y,t\eta _0')-\psi ^* (y,\eta ')+(x'-y')\cdot (\eta
   '-t\eta _0')
\end{split}
\end{equation}
We will soon drop the subscripts ``new'' when no confusion is
possible. This is again a nondegenerate phase function. The new
canonical relation $\kappa _\mathrm{new}:\, {\bf C}^{2n}_{y,y^*}\to
{\bf C}^{2n}_{x',t,{x'}^*,t^*}$ has the graph
\begin{equation}\label{fop.19}
\begin{split}
&\mathrm{graph\,}(\kappa _\mathrm{new})\\ &=
\{(y'+\partial _{\eta '}\psi ^*(y,\eta '),t,\eta '-t\eta _0', \eta
_0'\cdot \partial _{\xi '}\psi (y,t\eta _0')-
\eta _0'\cdot \partial _{\xi '}\psi^* (y,\eta ');\\
& y,-\partial _{y'}\psi (y,t\eta _0')+\partial _{y'}\psi ^*(y,\eta
')+\eta '-t\eta _0',-\partial _{y_n}\psi (y,t\eta _0')+\partial
_{y_n}\psi ^*(y,\eta '))
 \} .
\end{split}
\end{equation}
This graph is conic with respect to the dilations 
$$
{\bf R}_+\ni \lambda \mapsto (x',\lambda t,\lambda
{x'}^*,t^*;y,\lambda y^*)
$$
The restriction of the graph to $y_n=0$ is 
$$
\{ (y',t,\eta '-t\eta _0',0;y',0,\eta '-t\eta _0',-\partial _{y_n}\psi
(y',0,t\eta _0')+\partial _{y_n}\psi ^*(y',0,\eta ') )\} ,
$$
where
$$
\partial _{y_n}\psi (y',0,\xi ')=ir(y',0,\xi ')^{1/2},\ \partial _{y_n}\psi^* (y',0,\xi ')=-ir(y',0,\xi ')^{1/2}
$$
so the restriction is
\begin{equation}\label{fop.19.2}
\{(y',t,\eta '-t\eta _0',0;y',0,\eta '-t\eta
_0',-i(r^{\frac{1}{2}}(y',0,t\eta _0')+r^{\frac{1}{2}}(y',0,\eta '))) \}.
\end{equation}
If we take $\eta =t\eta _0'$ and use that $r^{\frac{1}{2}}$ is
homogeneous of degree 1 in the fiber variables, we get
\begin{equation}\label{fop.19.4}
\{ (y',t,0,0;y',0,0,-2itr^{\frac{1}{2}}(y',0,\eta _0')) \}.
\end{equation}
We assume, to fix the ideas, that $r(0,\eta _0')=1/4$.
This is the graph of a diffeomorphism
$$
\mathrm{neigh\,}(0,\partial \Omega )\times (-i{\bf R}^+_{y_n^*})\to
\mathrm{neigh\,}(0;\partial \Omega )\times {\bf R}^+_t.
$$

The tangent space of (\ref{fop.19}) at a point where $y_n=0$ is given by
\begin{equation}\label{fop.20}
\begin{split}
\{ (\delta _{y'}&+(\psi ^*)''_{\eta ',y_n}\delta _{y_n},\delta
_{t},\delta _{\eta '}-\delta _t\eta _0',\eta _0'\cdot (\psi ''_{\xi
  ',y_n}-(\psi ^*)''_{\eta ',y_n})\delta _{y_n}; 
\\ &
\delta _y,(-\psi ''_{y',y_n}+(\psi ^*)''_{y',y_n})\delta _{y_n}+\delta
_{\eta '}-\delta _t\eta _0',\\ &\hskip 1cm (-\psi ''_{y_n,y}+(\psi
^*)''_{y_n,y})\delta _y-\psi ''_{y_n,\xi '}\delta _t\eta _0'+(\psi
^*)''_{y_n,\eta '}\delta _{\eta '})
 \}.
\end{split}
\end{equation}
The projection onto the first component is injective as can be seen
exactly as in the proof of the property 2) stated after
(\ref{fop.14}). Now $\kappa _\mathrm{new}$ is a canonical relation
between spaces of the same dimension so we conclude that $\kappa
_\mathrm{new}$ is a canonical transformation locally near
each point of its graph. Combining this with the observation right
after (\ref{fop.19.4}), we get
\begin{prop}\label{fop1}
(\ref{fop.19}) is the graph of a bijective canonical transformation 
$$
\kappa _\mathrm{new}:\, \mathrm{neigh\,}((0;0,-i),{\bf C}^n_y\times
{\bf C}_{y^*}^n)\to \mathrm{neigh\,}((0,1;0),{\bf C}^n_{x',t}\times
{\bf C}^n_{{x'}^*,t^*}).
$$
The neighborhoods can be taken conic with respect to the actions ${\bf
  R}_+\ni \lambda \mapsto (y,\lambda y^*)$ and ${\bf
  R}_+\ni \lambda \mapsto  (x,\lambda t,\lambda {x'}^*,t^*)$ and $\kappa
_{\mathrm{new}}$ intertwines the two actions (so $\kappa
_\mathrm{new}$ is positively homogeneous of degree 1, with $y^*$ as
the fiber variables on the departure side and with $t,{x'}^*$ as the
fiber variables on the arrival side).
\end{prop}

\par Basically, the same exercise as the one leading to (\ref{fop.16})
shows that the image under $\kappa _\mathrm{new}$ of the complexified
conormal bundle, given by $y_n=0$, $(y^*)'=0$, is the zero section 
\begin{equation}\label{fop.20.5}
\{(x',t;({x'}^*,\, t^*)=0) \}.
\end{equation}

\par Consider the image of $T^*\partial \Omega \times i{\bf
  R}^-_{y^*_n}=\{(y,y^*);\,y',(y^*)'\in {\bf R}^{n-1},\, y_n= 0,\,
y_n^*\in i{\bf R}^- \}$ under $\kappa _\mathrm{new}$ and recall
(\ref{fop.19.2}).  If we restrict the attention to $t\in {\bf R}_+$,
so that $\eta '=(y^*)'+t\eta _0'\in {\bf R}^{n-1}$, we see
that $$y_n^*=-i(r^{\frac{1}{2}}(y',0,t\eta
_0')+r^{\frac{1}{2}}(y',0,\eta '))\in i{\bf R}^-.$$ Thus the image contains locally
$$\{(x',t,(x^*)',0);\, x',(x^*)'\in {\bf R}^{n-1},\, t\in {\bf R}^+
\},$$
which has the right dimension $2(n-1)+1$

\par Similarly, the image of $T^*\partial \Omega \times
\mathrm{neigh\,}(i{\bf R}^-_{y_n^*},{\bf C}_{y_n^*})$ is obtained by
dropping the reality condition on $t$ but keeping that on $\eta
'-t\eta _0'$, and we get 
\begin{equation}\label{fop.21}
\begin{split}
&\kappa _\mathrm{new}(T^*\partial \Omega \times \mathrm{neigh\,}(i{\bf
  R}^-_{y_n^*},{\bf C}_{y_n^*}))\\ &=
\{(x',t,{x'}^*,0);\, x',(x^*)'\in {\bf R}^{n-1},\, t\in
\mathrm{neigh\,}({\bf R}^+,{\bf C}) \}.
\end{split}
\end{equation}

\section{Some function spaces and their FBI-transforms}\label{fus}
\setcounter{equation}{0}

We continue to
work locally near a point $x_0$ where the boundary is analytic and we
use analytic coordinates $y$ centered at $x_0$ as specified in the
beginning of Section \ref{fop}.

\par
We start by defining some piecewise smooth I-Lagrange manifolds. 
\begin{itemize}
\item The cotangent space $T^*\Omega $ that we identify with 
$(\mathrm{neigh\,}(0)\cap {\bf R}^n_+)\times {\bf R}^n $.
\item The real conormal bundle $N^*\partial \Omega \subset T^*{\bf
    R}^n$. In the local coordinates $y$,
$$N^*\partial \Omega =\{(y,\eta )\in {\bf R}^{2n};\, y_n=0,\, \eta '=0
\} .$$ It will sometimes be convenient to write $N^*\partial \Omega
=\partial \Omega \times {\bf R}$ where of course the second
expression appeals to the use of special coordinates as above. More
invariantly, $N^*\partial \Omega $ is the inverse image
of the zero-section in $T^*\partial \Omega $ for the natural
projection map $\pi _{T^*\partial \Omega }:\, T^*_{\partial \Omega
}{\bf R}^n\to T^*\partial \Omega $.
\end{itemize}

\par We will also need some complex sets. 
\begin{itemize}
\item The complexified zero section in the complexification
  $\widetilde{T^*{\bf R}^n}={\bf C}^n_y\times {\bf C}^n_\eta $ is
  defined to be
$$
\mathrm{neigh\,}(0,{\bf C}^n)\times \{\eta =0\}\subset
{\bf C}_y^n\times {\bf C}^n_\eta .
$$
We denote it by ${\bf C}^n_y\times 0_\eta $ for short.
\item The complexification $\widetilde{N^*\partial \Omega }$ of
  $N^*\partial \Omega $ is defined to be 
$$
\{(y,\eta )\in {\bf C}^n_y\times {\bf C}^n_\eta ;\, y\in
\mathrm{neigh\,}(0,{\bf C}^n),\, y_n=0,\, \eta '=0 \}.
$$
\item The space $\pi ^{-1}(T^*\partial \Omega )$, where $\pi
  :T^*_{\partial \Omega}{\bf R}^n\otimes {\bf C}\to T^*\partial \Omega
  \otimes {\bf C}$ is the natural projection and $\otimes {\bf C}$
  indicates fiberwise complexification. In special coordinates it is
  $\{(y,\eta );\, (y',\eta ')\in {\bf R}^{2(n-1)},\, y_n=0,\, \eta
  _n\in {\bf C} \}$. We will denote it by $T^*\partial \Omega \times
  {\bf C}$ or $T^*\partial \Omega \times {\bf C}_{\eta _n}$ for
  simplicity. It contains the subset $T^*\partial \Omega \times {\bf
    C}^-_{\eta _n}$ (easy to define invariantly), where ${\bf C}^-$ is
  the open lower half-plane. Notice that
$$
T^*\partial \Omega \times \partial {\bf C}_-=T^*\partial \Omega \times
{\bf R}=T_{\partial \Omega }^*{\bf R}^n.
$$
\item The piecewise smooth (Lipschitz) manifold 
$$F=\overline{T^*\Omega }\cup (T^*\partial \Omega \times {\bf C}_{\eta _n}^-).$$
Notice that the two components to the right have $T^*_{\partial \Omega }{\bf
  R}^n$ as
their common boundary.
\item The piecewise smooth (Lipschitz) manifold $({\bf C}_y^n\times
  0_\eta )\cup \widetilde{N^*\partial \Omega }$ where the two
  constituents contain $\widetilde{\partial \Omega }\times 0_\eta
  $. Here $\widetilde{\partial \Omega } $ denotes a complexification of the
  boundary (near $x_0$).
\end{itemize}

Let
\begin{equation}\label{fus.1}
Tu(z;h)=Ch^{-\frac{3n}{4}}\int_{{\bf R}^n} e^{\frac{i}{h}\phi
  (z,y)}u(y)dy,\ z\in {\bf C}^n,
\end{equation}
be a standard FBI transform (\cite{Sj82}), sending distributions with
compact support on ${\bf R}^n$ to holomorphic functions on (in general
some subdomains of) ${\bf C}^n$. For simplicity we let $\phi $ be a
holomorphic quadratic form so that $T$ can also be viewed as a generalized
Bargmann transform and a metaplectic Fourier integral operator (see
for instance \cite{Sj90}). We
work under the standard assumptions
\begin{equation}\label{fus.2}
\Im \phi ''_{y,y}>0,\quad \det \phi ''_{z,y}\ne 0.
\end{equation}
We let $C>0$ be the unique positive constant for which $T:\,L^2({\bf
  R}^2)\to H_{\Phi _0}({\bf C}^n)$ is unitary, where 
\begin{equation}\label{fus.3}
\Phi _0(z)=\sup_{y\in {\bf R}^n}-\Im \phi (z,y)=-\Im \phi (z,y(z))
\end{equation}
is a strictly pluri-subharmonic (real) quadratic form on ${\bf C}^n$
and $H_{\Phi _0}$ is the complex Hilbert space $\mathrm{Hol\,}({\bf
  C}^n)\cap L^2(e^{-2\Phi _0/h}L(dz))$, $L(dz)$ denoting the Lebesgue
measure on ${\bf C}^n\simeq {\bf R}^{2n}$. { In Example
  \ref{fus1} we will discuss the case of a standard Bargmann transform
  where $\phi (z,y)=i(z-y)^2/2$ and compute the various weights $\Phi
  _0,\,\Phi _1,...$.}

\par Let 
\begin{equation}\label{fus.4}
\kappa _T:\, {\bf C}^{2n}\ni (y,-\phi '_{y}(z,y))\mapsto (z,\phi
'_z(z,y))\in {\bf C}^{2n}
\end{equation}
be the complex (linear) canonical transformation associated to $T$ and
let $\Lambda _{\Phi _0}=\{(z,\frac{2}{i}\frac{\partial \Phi
  _0}{\partial z}(z));\, z\in {\bf C}^n \}$ be the R-symplectic\footnote{
i.e. symplectic with respect to $\Re \sigma $, where $\sigma =d\zeta \wedge
dz $ is the complex symplectic form} and I-Lagrangian\footnote{i.e. Lagrangian with respect to $\Im \sigma $}
manifold of ${\bf C}^{2n}$, actually a real-linear subspace since
$\phi $ is quadratic. Then we know that 
\begin{equation}\label{fus.5}
\Lambda _{\Phi _0}=\kappa _T({\bf R}^{2n}).
\end{equation}
More explicitly,
\begin{equation}\label{fus.6}
\kappa _T^{-1}(z,\frac{2}{i}\frac{\partial \Phi _0}{\partial
  z})=(y(z),\eta (z))\in {\bf R}^{2n},
\end{equation}
where $y(z)$ appeared in (\ref{fus.3}).

\par In this paper, we shall deal with FBI-transforms and $H_\Phi $
locally and we recall some definitions and facts from
\cite{Sj82}. The local $H_\Phi $-spaces are defined in Chapter 1 of
that work:

\par Let $\Omega \subset {\bf C}^n$ be an open set, $\Phi :\Omega \to
{\bf R}$ a continous function. A function $u(z;h)$ on $\Omega \times
]0,h_0]$, $0<h_0\le 1$, is said to belong to $H_\Phi
^\mathrm{loc}(\Omega )$ if
\begin{equation}\label{hphi.1}
u \hbox{ is holomorphic in }z\hbox{ for each }h\in ]0,h_0],
\end{equation}
\begin{equation}\label{hphi.2}\begin{split}&\hbox{For every compact set
  }K\subset\Omega \hbox{ and every }\epsilon >0,\\ &\hbox{there exists a
  constant }C=C_{\epsilon ,K}>0,\hbox{ such that}\\
&|u(z;h)|\le C{e^{\frac{1}{h}(\Phi (z)+\epsilon )},}\
(z,h)\in K\times ]0,h_0].
\end{split}
\end{equation}
In general, we shall not {distinguish} between two elements $u,v\in
H_{\Phi }^\mathrm{loc}(\Omega )$, if their difference is exponentially small
relative to $e^{\Phi (z)/h}$. More precisely, if $u,v\in
H_{\Phi }^\mathrm{loc}(\Omega )$ we say that they are equivalent
($u\sim v$) if there exists a continous function $\widetilde{\Phi
}<\Phi $ on $\Omega $, such that 
\begin{equation}\label{hphi.3}
u-v\in H_{\widetilde{\Phi }}^\mathrm{loc}(\Omega ).
\end{equation}
This is clearly an equivalence relation and sometimes we do not
distinguish between $H_{\Phi }^\mathrm{loc}(\Omega )$ and the
corresponding set of equivalence classes. 

It will also be convenient to work with germs of $H_\Phi
$-functions. If $z_0\in {\bf C}^n$ and $\Phi \in
C(\mathrm{neigh\,}(z_0,{\bf C}^n);{\bf R})$, then by definition, an
element $u\in H_{\Phi ,z_0}$ is an element $u\in H_\Phi
^\mathrm{loc}(\Omega )$, where $\Omega $ is a neighborhood of
$z_0$. We say that $u,v\in H_{\Phi ,z_0}$ are equivalent, $u\sim v$,
if they are equivalent in $H_{\Phi }^\mathrm{loc}(W)$ for some
neighborhood $W$ of $z_0$.

The corresponding microlocal version of FBI-transforms is then the
following: Let $\phi \in
\mathrm{Hol\,}(\mathrm{neigh\,}((z_0,y_0),{\bf C}^{2n})$ satisfy
(\ref{fus.2}) at the point $(z_0,y_0)$. Also assume that 
\begin{equation}\label{hphi.4}
\eta _0':=-\phi '_y(z_0,y_0)\in {\bf R}^n.
\end{equation}
Then (cf.(\ref{fus.3})) we can {define} $\Phi _0\in C^\infty
(\mathrm{neigh\,}(z_0,{\bf C}^n);{\bf R})$ by 
\begin{equation}\label{hphi.5}
\Phi _0(z)=\sup_{y\in \mathrm{neigh\,}(y_0,{\bf R}^n)}(-\Im \phi (z,y)),
\end{equation}
and $\Phi _0$ becomes strictly plurisubharmonic. As after
(\ref{fus.3}), we can define the canonical transformation 
$$
\kappa _T:\mathrm{neigh\,}((y_0,\eta _0),{\bf C}^{2n})\to
\mathrm{neigh\,}((z_0,\zeta _0),{\bf C}^{2n}),
$$
where $\zeta _0=\phi '(z_0,y_0)=\frac{2}{i}\partial _z\Phi_0 (z_0)$
and we have natural local versions of (\ref{fus.5}). If $u\in {\cal
  D}'(\mathrm{neigh\,}(y_0,{\bf R}^n))$ is independent of $h$ or more
generally $h$-dependent but of temperate growth in ${\cal D}'$ as a
function of $h$, then by throwing in a cutoff $\chi \in C_0^\infty
(\mathrm{neigh\,}(y_0,{\bf R}^n))$ with $y_0\not\in
\mathrm{supp\,}(1-\chi )$ into the formula (\ref{fus.1}), we can
define $Tu\in H_{\Phi _0,z_0}$ up to $\sim$. 

\par We now return to the FBI-transform (\ref{fus.1}) with
quadratic phase.
Let \begin{equation}\label{fus.6.5}\Phi
_1^\mathrm{ext}(z)=\sup_{y\in \partial {\bf R}_+^n}-\Im \phi
(z,y)=-\Im \phi (z,\widetilde{y}(z)), \end{equation}
where $\widetilde{y}(z)=(\widetilde{y}'(z),0)$ and $\widetilde{y}'(z)$
is the unique point of maximum in ${\bf R}^{n-1}$ of $y'\mapsto -\Im
\phi (z,y',0)$.
If $\mathrm{supp\,}u\subset \{y\in {\bf R}^n;\, y_n\ge 0 \}$, then
$Tu\in H_{\Phi _1}^\mathrm{loc}$, where
\begin{equation}\label{fus.7}
\Phi _1(z)=\sup_{y\in{\bf R}_+^n}-\Im \phi (z,y)=\begin{cases}
\Phi _0(z),\hbox{ if }y_n(z)\ge 0,\\
\Phi _1^\mathrm{ext}(z),\hbox{ if }y_n(z)\le 0.
\end{cases} 
\end{equation}
Notice that 
\begin{itemize}
\item $-\Im \partial _{y_n}\phi (z,\widetilde{y}(z))\ge 0$ in the
  first case,
\item $-\Im \partial _{y_n}\phi (z,\widetilde{y}(z))\le 0$ in the
  second case.
\end{itemize}
Moreover,
$$
\frac{2}{i}\frac{\partial \Phi _1}{\partial
  z}(z)=\frac{2}{i}(\frac{\partial }{\partial z}(-\Im \phi )
)(z,\widetilde{y}(z))=\phi '_z(z,\widetilde{y}(z))
$$
and
$
\widetilde{\eta }(z)=-\phi '_y(z,\widetilde{y}(z))
$
satisfies $\widetilde{\eta }'(z)\in {\bf R}^{n-1}$. When $\Phi
_1(z)=\Phi _1^\mathrm{ext}(z)$ we have
\begin{equation}\label{fus.8}\widetilde{\eta }'(z)\in {\bf
    R}^{n-1},\ \Im \widetilde{\eta }_n(z)\le 0. \end{equation}
This means that
$$
\Lambda _{\Phi _1^\mathrm{ext}}=\kappa _T(T^*\partial \Omega \times
{\bf C}_{\eta _n}),
$$
and that
\begin{equation}\label{fus.9}
\Lambda _{\Phi _1}=\kappa_T (F),
\end{equation}
near $(z,\frac{2}{i}\partial _z\Phi _1(z))$ in case of strict
inequality in (\ref{fus.8}). Here the Lipschitz manifold
$F$ was defined above,
\begin{equation}\label{fus.10}
F=\overline{T^*(\Omega )}\cup \{(y',0;\eta ',\eta _n);\, (y',\eta
')\in T^*\partial \Omega ,\ \Im \eta _n\le 0 \}.
\end{equation}
The second component is a union of
complex half-lines, consequently in the region where $\Phi _1<\Phi_0 $,
$\Lambda _{\Phi _1}$ is a union of complex half-lines. If we project these
lines to the complex $z$-space we get a foliation of ${\bf C}^n_z$
into complex half-lines and the restriction of $\Phi _1$ to each of
these is harmonic. We have the corresponding local statements.

\par If $y_0=(y_0',0)\in \partial {\bf R}_+^n,\ (y_0,\eta _0)\in F$
and $z_0=\pi _z\kappa _T(y_0,\eta _0)$, then for $u\in {\cal
  D'}(\mathrm{neigh\,}(y_0,{\bf R}^n))$ with $\mathrm{supp\,}u\subset
{\bf R}_+^n$, we have that $Tu$ is well-defined up to equivalence in
$H_{\Phi _1,z_0}$.

\par We introduce the real hyperplane
$$
H=\pi _z\kappa _T(T^*_{\partial \Omega }{\bf R}^n),
$$
which is the common boundary of the two half-spaces
$$
H_+=\pi _z\kappa _T(T^*\Omega ),
$$
$$
H_-=\pi _z\kappa _T(\{(y',0;\eta);\, (y',\eta
')\in T^*\partial \Omega ,\ \Im \eta _n< 0 \}).
$$
Here, $\pi _z:{\bf C}_z^n\times {\bf C}^n_\zeta \to {{\bf C}_z^n}$ is the
natural projection.
We have
\begin{equation}\label{fus.11}
\Phi _0-\Phi _1\begin{cases}=0\hbox{ in }H_+,\\
\asymp \mathrm{dist\,}(z,H)^2\hbox{ in }H_-.
\end{cases}
\end{equation}

\par Similarly, recall the definition of the complexified normal
bundle $\widetilde{N^*\partial \Omega }$ at the beginning of this
section. It is a ${\bf C}$-Lagrangian manifold.\footnote{i.e.~a
  holomorphic manifold which is Lagrangian for the complex symplectic
  form $\sigma $.} We have 
$\kappa _T(\widetilde{N^*\partial \Omega })=\Lambda _{\Phi _3}$, where
$\Phi _3$ is pluriharmonic and given by 
$$
\Phi _3(z)=\mathrm{v.c.}_{y'\in {\bf C}^{n-1}}(-\Im \phi (z,y',0)),
$$
$\mathrm{v.c.}$ $=$ ``critical value of''.

\par Similarly $\kappa _T({\bf C}_y^n\times 0_\eta )$ (with the notation
from the beginning of this section) is of the form $\Lambda _{\Phi
  _4}$, where
$$
\Phi _4(z)=\mathrm{v.c.}_{y\in {\bf C}^n}(-\Im \phi (z,y)).
$$
The complex zero-section ${\bf C}_y\times 0_\eta $ and $T^*{\bf
  R}^n$ intersect transversally along the real zero-section ${\bf
  R}^n_y\times 0_\eta $. Correspondingly, we check that 
\begin{equation}\label{fus.12}
\Phi _0(z)-\Phi _4(z)\asymp \mathrm{dist\,}(z,\pi _z\circ \kappa
_T({\bf R}^n\times 0_\eta ))^2.
\end{equation}

\par Similarly, 
\begin{equation}\label{fus.13}
\Phi _1^\mathrm{ext}(z)-\Phi _3(z)\asymp \mathrm{dist\,}(z,\pi _z\circ
\kappa _T((\partial \Omega \times 0)\times {\bf C}^*_{\eta _n}))^2,
\end{equation}
where $\partial \Omega \times 0$ denotes the zero section in
$T^*\partial \Omega $, so that 
$$(\partial \Omega \times 0)\times {\bf C}^*_{\eta _n}=N^*\partial
\Omega \otimes {\bf C}$$ is the fiber-wise complexification of
$N^*\partial \Omega $. (Here we work locally near $y=0$.)

\par Let $u$ be real-analytic in a neighborhood of $y_0\in \partial
{\bf R}_+^n$ and consider
\begin{equation}\label{fus.14}
v(z)=T(1_\Omega u)(z),
\end{equation}
where we restrict the attention to $z\in {\bf C}^n$ such that the
critical point $y_{\Phi _4}(z)$ in the definition of $\Phi _4(z)$
belongs to a small complex neighborhood of $y_0\in \partial {\bf R}_+^n$ or
equivalently to $z\in {\bf C}^n$ in a small neighborhood of $\kappa
_T(\{ y_0\}\times 0_\eta )$. By the ``m\'ethode du col''{\footnote{ In
  the present situation this means that we deform the integration
  contour so that its boundary remains in the the complex hyperplane
  $y_n=0$ and so that the supremum of the modulus of the exponential in the integral
  is as small as possible, see \cite{Sj82} and Example \ref{fus1} below.}} we
see that $v\in H_{\Phi _5}^\mathrm{loc}$, where first of all $\Phi
_5\le \Phi _1$ and further,
\begin{equation}\label{fus.15}
\Phi _5(z)=\Phi _4(z),\hbox{ when }\begin{cases}\Re y_{\Phi _4}(z)\in
  \Omega \hbox{ and}\\
|\Im y_{\Phi _4}(z)|\ll \mathrm{dist\,}(\Re y_{\Phi _4}(z),\partial
\Omega ),
\end{cases}
\end{equation}
\begin{equation}\label{fus.16}
  \Phi _5(z)=\Phi _3(z),\hbox{ when }\begin{cases}\Re y_{\Phi _4}(z)
    \not\in
    \Omega \hbox{ and}\\
    |\Im y_{\Phi _4}(z)|\ll \mathrm{dist\,}(\Re y_{\Phi _4}(z),\partial
    \Omega ).
\end{cases}
\end{equation}
Actually, in the last case we can relax the condition that
$y_{\Phi _4}(z)$ belongs to a small ($u$-dependent) neighborhood of
$\overline{\Omega }$. The appropriate restriction is then that the
critical point $y_{\Phi _3}(z)\in \widetilde{\partial \Omega }$ in the
definition of $\Phi _3$ belongs to a small ($u$-dependent)
neighborhood of $\partial \Omega $.
{
\begin{ex}\label{fus1}{\rm
Let $$\phi (z,y)=\frac{i}{2}(z-y)^2.$$ The associated canonical
transformation is given by
$$
\kappa _ T:\, (y,i(z-y))\mapsto (z,i(z-y)),
$$
or more explicitly, by
\begin{equation}\label{fus.17}
\kappa _T:(y,\eta )\mapsto (y-i\eta ,\eta ).
\end{equation}
We get
\begin{equation}\label{fus.18}
\Phi _0(z)=\frac{1}{2}(\Im z)^2,\ \Lambda _{\Phi _0}:\, \zeta =-\Im
z,\ y(z)=\Re z.
\end{equation}
\begin{equation}\label{fus.19}
\Phi _1^\mathrm{ext}(z)=\frac{1}{2}(\Im z)^2-\frac{1}{2}(\Re z_n)^2,
\end{equation}
\begin{equation}\label{fus.20}
\Phi _1(z)=\begin{cases}\frac{1}{2}(\Im z)^2,\ &\Re z_n\ge 0,\\
\frac{1}{2}(\Im z)^2 -\frac{1}{2}(\Re z_n)^2,\ &\Re z_n\le 0.
\end{cases}
\end{equation}
\begin{equation}\label{fus.21}
\begin{split}
F=&\{ (y,\eta )\in {\bf R}^{2n};\, y_n\ge 0 \}\cup \\
&\{(y,\eta )\in {\bf
  R}^n\times {\bf C}^n;\, y_n\le 0,\ \eta '\in {\bf R}^{n-1},\ \Im
\eta _n\le 0 \} .
\end{split}
\end{equation}
\begin{equation}\label{fus.22}
H=\{ z\in {\bf C}^n;\, \Re z_n=0 \},\ H_\pm =\{z\in {\bf C}^n;\, \pm
\Re z_n>0 \} .
\end{equation}
\begin{equation}\label{fus.23}
\Phi _3(z)=\frac{1}{2}(\Im z_n)^2-\frac{1}{2}(\Re z_n)^2.
\end{equation}
\begin{equation}\label{fus.24}
\Phi _4(z)=0.
\end{equation}
(\ref{fus.13}) reads,
$$
\Phi _1^\mathrm{ext}(z)-\Phi _3(z)\asymp (\Im z')^2,
$$
since
$$
\pi _z\kappa _T((\partial \Omega \times 0)\times {\bf C}_{\eta
  _n}^*)=\{z;\, \Im z'=0 \}.
$$

\par The definition of $\Phi _5$ depends on the domain to which $u$
extends holomorphically. Assume $u$ extends to a bounded holomorphic
function in $\{y\in {\bf C}^n;\, |\Im y'|<\epsilon ,\ |\Im
y_n|<\epsilon ,\ \Re y_n>0 \}$. (Later on, we will need only that $u$
is defined for $|\Im y|<\epsilon $, $0<\Re y_n<1/C$ and the
conclusion about $\Phi _5$ will then hold in a suitable smaller
domain.)  Let us first discuss the theoretical case when $\epsilon
=+\infty $. The function $y\mapsto \phi (z,y)$ has a unique critical
point $y=z$ and this point is a saddle point for
$$
-\Im \phi (z,y)=\frac{1}{2}(\Im z-\Im y)^2-\frac{1}{2}(\Re z-\Re y)^2.
$$
We want to deform the contour ${\bf R}_{y'}^{n-1}\times [0,+\infty [$
to a new contour with boundary contained in $y_n=0$, passing through
the critical point and remaining in the region $-\Im \phi (z,y)\le
0$. As for the $y'$-variables we make a simple translation from ${\bf
  R}_{y'}^{n-1}$ to $z'+{\bf R}_{y'}^{n-1}$, so we can concentrate on
the one-dimensional problem for the $y_n$-variable. 
\begin{itemize}
\item When $\Re z_n\ge
|\Im z_n|$ this is possible and we get $\Phi _5=0$. Indeed the contour
will start at $0$ which is in one of the two valleys, emanating from the saddle point, then reach the
saddle point $y_n=z_n$ and can then continue in the other valley and
join the positive real axis. 
\item When $0< \Re z_n<|\Im z_n|$, the point $0$ is situated on one of the
  mountain sides above the saddle point and there is no point in
  descending to that point. Assume, to fix the ideas, that $\Re
  z_n>0$, then we can take as contour the level curve of
  $\frac{1}{2}\left( (\Im z_n-\Im y_n)^2-(\Re z_n-\Re y_n)^2\right)$
  starting at $y_n=0$ which enters the upper half-plane and follow it
  until it hits $]0,+\infty [$, and then follow the positive half-axis
  in the increasing direction. We get $\Phi
  _5(z)=\frac{1}{2}\left((\Im z_n)^2-(\Re z_n)^2 \right)$.
\item When $\Re z_n\le 0$, we can take the contour $[0,+\infty [$ in
  the $y_n$-plane and reach the same conclusion.
\end{itemize}

\par When $0<\epsilon <+\infty  $, we can still describe $\Phi _5$
everywhere on ${\bf C}^n$, but content ourseleves with the observation
that the above description remains valid in the union of the following regions: 
\begin{itemize}
\item $|\Im z'|< \epsilon $ and $\Re z_n<0$,
\item $|\Im z'|< \epsilon $, $|\Im z_n|< \epsilon $  and $\Re z_n\ge0$.
\end{itemize}

}
\end{ex}}

\section{Expressing $M$ with the help of FBI-transforms}\label{fbi}
\setcounter{equation}{0}
From now on we work with $M_\mathrm{new}$, $\Phi _{M_\mathrm{new}}$,
$\kappa _\mathrm{new}$ and we drop  the corresponding subscript ``new''. Then
(\ref{fop.17}) reads
\begin{equation}\label{fbi.1}
Mq(x',t)=\frac{1}{(2\pi h)^{n-1}}\iint e^{\frac{i}{h}{\Phi_M} (x',t,y,\eta
')} a(y,-\eta ';h)a(y,t\eta _0';h) q(y)dyd\eta '.
\end{equation}
with ${\Phi_M} $ given in (\ref{fop.18}). 

We want to express $Mq$ with the help
of $Tq$, where $T$ is as in (\ref{fus.1}) and we start by recalling
some general facts about metaplectic Fourier integral operators of
this form, following \cite{Sj82} for the local theory, and \cite{Sj90}
for the simplified global theory in the metaplectic frame work
(i.e. all phases are quadratic and all amplitudes are constant). To
start with, we weaken the assumptions on the quadratic phase in $T$ and
assume only that $\phi (x,y)$ is a holomorphic quadratic form on ${\bf
  C}^n\times {\bf C}^n$,
satisfying the second part of (\ref{fus.2}):
\begin{equation}\label{fbis.2}
\det \phi ''_{x,y}(x,y)\ne 0.
\end{equation}
To $T$ we can still associate a linear canonical transformation $\kappa_
T$ as in (\ref{fus.4}).
Let $\Phi _1$, $\Phi _2$ be plurisubharmonic quadratic forms on ${\bf
  C}^n$ related by 
\begin{equation}\label{fbis.3}
\Lambda _{\Phi _2}=\kappa _T(\Lambda _{\Phi _1})
\end{equation}
Then we can define $T:H_{\Phi _1}\to H_{\Phi _2}$ as a bounded
operator as in (\ref{fus.1}) with the modification that ${\bf R}^n$
should be replaced by a so called good contour, which is an affine
subspace of ${\bf C}^n$ of real dimension $n$, passing through the
nondegenerate critical point $y_c(x)$ the function
\begin{equation}\label{fbis.4}
y\mapsto -\Im \phi (x,y)+\Phi _1(y)
\end{equation}
and along which this function is $\Phi _2(x)-(\asymp
|y-y_c(x)|^2)$. (Actually in this situation it would have been better
to replace the power $h^{-3n/4}$ by $h^{-n/2}$ since we would then get
a uniform bound on the norm.)

\begin{remark}\label{fbis1}{\rm Recall also that if only $\Phi _1$ is
    given as above, the existence of a quadratic form $\Phi _2$ as in
    (\ref{fbis.3}) is equivalent to the fact that (\ref{fbis.4}) has a
    nondegenerate critical point and the plurisubharmonicity of $\Phi
    _2$ is equivalent to the fact that the signature of the critical
    point is $(n,-n)$ (which represents the maximal number of negative
    eigenvalues of the Hessian of a plurisubharmonic quadratic
    form). This in turn is equivalent to the existence of an affine
    good contour as above.}\end{remark}

In this situation $T:H_{\Phi _1}\to H_{\Phi _2}$ is bijective with the inverse
\begin{equation}\label{fbi.2}
Sv(y)=T^{-1}v(y)=\widetilde{C}h^{-\frac{n}{4}}\int e^{-\frac{i}{h}\phi
(z,y)}v(z)dz,
\end{equation}
which can be realized the same way with a good contour and here the
constant $\widetilde{C}$ does not depend on the choice of $\Phi _j$, $j=1,2$.
\begin{remark}\label{fbis2}{\rm
Let us introduce the formal adjoints of $T$ and $S$,
$$T^\mathrm{t}v(y)=Ch^{-\frac{3n}{4}}\int_{{\bf R}^n} e^{\frac{i}{h}\phi
  (z,y)}v(z)dz,\ y\in {\bf C}^n,
$$
$$
S^\mathrm{t}u(z)=\widetilde{C}h^{-\frac{n}{4}}\int e^{-\frac{i}{h}\phi
(z,y)}u(y)dy.
$$
Let $\Psi _1$, $\Psi _2$ be pluri-subharmonic quadratic forms such
that $\kappa _{S^\mathrm{t}}(\Lambda _{\Psi _1})=\Lambda _{\Psi
  _2}$. Then as above, $T^\mathrm{t}:H_{\Psi _2}\to H_{\Psi _1}$, $S^\mathrm{t}:H_{\Psi _1}\to H_{\Psi _2}$ are
bijective and $S^\mathrm{t}=\mathrm{const.}(T^\mathrm{t})^{-1}$. We claim that
$S^\mathrm{t}$ is the inverse of $T^\mathrm{t}$. In fact, this
statement is independent of the choice of $\Phi _j$,
$\Psi _j$ as above and we
can choose them to be pluri-harmonic in such a way that $\Lambda
_{\Phi _j}$ intersects $\Lambda _{-\Psi _j}$ transversally for one
value of $j$ and then automatically for the other value. Then for $j=1,2$ we can define
$$\langle u|v\rangle = \int _{\gamma_j} u(x)v(x)dx,$$
for $u\in H_{\Phi _j}$, $v\in H_{\Psi _j}$ (or rather for functions
that are ${\cal O}(e^{\Phi _j/h})$ and $e^{\Psi _j/h}$ respectively,
the space of such functions is of dimension 1 which suffices for our
purposes) if we let $\gamma _j$ be a good contour for $\Phi _j+\Psi
_j$. For $u={\cal O}(e^{\Phi _2/h})$, $v={\cal O}(e^{\Psi _2/h})$ non
zero, we have
$$
0\ne \langle u|v\rangle=\langle TSu|v\rangle =\langle
Su|T^\mathrm{t}v\rangle
=\langle u|S^\mathrm{t}T^{t}v\rangle 
$$
and knowing already that $S^\mathrm{t}T^\mathrm{t}$ is a multiple of the
identity, we see that it has to be equal to the identity. 
}\end{remark}

Now return to the discussion of an FBI-transform $T$ whose phase
satisfies (\ref{fus.2}). When letting $T$ act on suitable $H_\Phi
$-spaces it has the inverse $S$ in (\ref{fbi.2}). However, if we let
$T$ act on $L^2({\bf R}^n)$ so that $Tu\in H_{\Phi _0}$ (with $\Lambda
_{\Phi _0}=\kappa _T({\bf R}^{2n})$), the best possible contour in
(\ref{fbi.2}) is
$$
\Gamma (y)=\{z\in {\bf C}^n;\, y(z)=y \}.
$$
This follows from the property
\begin{equation}\label{fbi.3}
\Phi _0(z)+\Im \phi (z,y)\asymp \mathrm{dist\,}(z,\Gamma (y))^2\asymp |y(z)-y|^2,
\end{equation}
so $\Phi _0(z)+\Im \phi (z,y)=0$ on $\Gamma (y)$ and
$e^{-\frac{i}{h}\phi (z,y)+\frac{1}{h}\Phi _0(z)}$ is bounded
there. This is not sufficient for a straight forward definition of
$Sv(y)$, $v\in H_{\Phi _0}$ since we would need some extra exponential
decay along the contour near infinity, but it does suffice to give a
precise meaning up to exponentially small errors of the formula
\begin{equation}\label{fbi.4}
\widetilde{T}u=(\widetilde{T}S)Tu,
\end{equation} 
in a local situation,
where $\widetilde{T}:L^2\to H_{\widetilde{\Phi }_0}$ is a second
FBI-transform and where $\widetilde{T}S:\, H_{\Phi _0}\to
H_{\widetilde{\Phi }_0}$ is defined by means of a good contour.
\begin{prop}\label{fbis3}
  Let $(y_0,\eta _0)\in {\bf R}^{2n}$, $(z_0,\zeta _0)=\kappa
  _T(y_0,\eta _0)$, $(w_0,\omega _0)=\kappa _{\widetilde{T}}(y_0,\eta
  _0)$. We realize $Tu$, $\widetilde{T}u$, $\widetilde{T}Su$ (modulo
  exponentially small terms) in $H_{\Phi _0,z_0}$, $H_{\widetilde{\Phi
    } _0,w_0}$, $H_{\widetilde{\Phi } _0,w_0}$ respectively, by
  choosing good contours restricted to neighborhoods of $y_0$, $y_0$,
  $z_0$ respectively. Then (\ref{fbi.4}) holds (modulo an
  exponentially small error) in $H_{\widetilde{\Phi } _0,w_0}$. Here
  $u\in {\cal D}'({\bf R}^n)$ is either independent of $h$ or of
  temperate growth in ${\cal D}'({\bf R}^n)$ as a function of $h$.  
\end{prop} 
\begin{proof}
The left hand side of (\ref{fbi.4}) is 
$$\mathrm{Const.\,}h^{-\frac{3n}{4}-n}\iiint
e^{\frac{i}{h}(\widetilde{\phi }(w,x)-\phi (z,x)+\phi (z,y))}u(y) dy
dz dx$$
where the composed contour is good,
and all good contours being homotopic, we can write it as
$$
\widetilde{C}h^{-\frac{3n}{4}}\int \left( \mathrm{Const.\,}h^{-n}\iint
  e^{\frac{i}{h}(-\phi (z,x)+\phi (z,y))}e^{\frac{i}{h}\widetilde{\phi
    }(w,x)} dxdz \right)u(y) dy.
$$
The expression in the big parenthesis is nothing but
$T^\mathrm{t}S^\mathrm{t}(e^{\frac{i}{h}\widetilde{\phi }(w,\cdot
  )})(y)$, which by Remark \ref{fbis2} is equal to
$e^{\frac{i}{h}\widetilde{\phi }(w,y)}$ and (\ref{fbi.4}) follows. (In
the proof we have chosen not to spell out the various exponentially
small errors due to the fact that the integration contours are
confined to various small neighborhoods of certain points.)
\end{proof}

We now return to the operator $M$ in (\ref{fbi.1}). Choose adapted
analytic coordinates centered at $x_0$ as in the beginning of Section
\ref{fop}. In that section (cf (\ref{fop.21})) we have seen that there
is a well defined canonical transformation $\kappa _M$ from a
neighborhood of $(0,0,-i)\in {\bf C}^{2n}_{y,\eta }$ to a neighborhood
of $(0,1,0,0)$ in ${\bf C}_{x'}^{n-1}\times {\bf C}_t\times {\bf
  C}_{{x'}^*}^{n-1}\times {\bf C}_{t^*}$ mapping $T^*\partial \Omega
\times i{\bf R}_-$ to ${\bf R}_{x'}^{n-1}\times {\bf R}_t\times {\bf
  R}_{{x'}^*}^{n-1}\times \{ t^*=0 \}$. This means that we have a
microlocal description of $Mq$ near $(0,1,0,0)$ and not a local one
near $x'=0$, {$t=1$}. We shall therefore microlocalize in $(x',{x'}^*)$
by means of an FBI-transform in the $x'$-variables.

Let
\begin{equation}\label{fbi.10}
  \widehat{T}u(w')=\widehat{C}h^{\frac{1-n}{2}}\int_{{\bf
      R}^{n-1}} e^{\frac{i}{h}\widehat{\phi }(w',x')}u(x')dx',\ w'\in
  {\bf C}^{n-1}
\end{equation}
be a second FBI-transform as in (\ref{fus.1}) though acting in $n-1$
variables and with a different normalization. Assume (to fix the ideas) that
\begin{equation}\label{fbi.11}
  \kappa _{\widehat{T}}({\bf C}^{n-1}_{x'}\times \{0 \})={\bf
    C}_{w'}^{n-1}\times \{0 \} .
\end{equation}
Then 
\begin{equation}\label{fbi.12}
\kappa _{\widehat{T}}(T^*{\bf R}^{n-1})=\Lambda _{\widehat{\Phi }_0},
\end{equation}
where $\widehat{\Phi }_0$ is a strictly plurisubharmonic quadratic
form. 

\par By slight abuse of notation we also let $\widehat{T}$ act on
functions of $n$ variables by
$$
\widehat{T}(u)(w',t)=(\widehat{T}u(\cdot ,t))(w').
$$

The presence of $\widehat{T}$ leads to a formula for $\widehat{T}M$
that is simpler than the one for $M$ in
(\ref{fbi.1}). 
\[
\begin{split}
\widehat{T}Mq(w',t)&=\widehat{T}\left( e^{-\frac{i}{h}(\cdot )\cdot
    t\eta _0'}K^\mathrm{t}qK\left(e^{\frac{i}{h}(\cdot )\cdot t\eta _0'}
  \right) \right)(w') =\\
\widehat{C}h^{\frac{1-n}{2}}&\iiint e^{\frac{i}{h}(\widehat{\phi
  }(w',\widetilde{x}')-\widetilde{x}'\cdot t\eta
  _0')}K(y,\widetilde{x}')q(y)K(y,x')e^{\frac{i}{h}x'\cdot t\eta
  _0'}dx'dyd\widetilde{x}'\\
&=\int K\left(e^{\frac{i}{h}(\widehat{\phi }(w',\cdot )-(\cdot )\cdot
    t\eta _0')} \right)(y)q(y)K\left(e^{\frac{i}{h}(\cdot )\cdot t\eta _0'}
\right)(y)dy .
\end{split}
\]
Up to exponentially small errors we have (cf.~(\ref{fop.4.5}))
$$
K\left(e^{\frac{i}{h}(\cdot )\cdot t\eta _0'}
\right)(y)=e^{\frac{i}{h}\phi (y,t\eta _0')}a(y,t\eta _0';h)
$$
and
$$
K\left( e^{\frac{i}{h}\widehat{\phi }(w',\cdot )}
\right)(y)=e^{\frac{i}{h}\widetilde{\psi }(w',t\eta _0',y)}b(w',y,t\eta _0';h),
$$
where $b$ is an elliptic analytic symbol of order $0$ and
{$\widetilde{\psi } $} is the solution of the
following eikonal equation in $y$,
$$
\partial _{y_n}\widetilde{\psi  }=ir(y,\partial _{y'}{\widetilde{\psi }}
)^{\frac{1}{2}},\quad {{\widetilde{\psi
    }}}_{\vert y_n=0}=\widehat{\phi }(w',y')-y'\cdot t\eta _0'.
$$
Thus, up to exponentially small errors, we get for $q\in L^\infty
(\Omega )$,
\begin{equation}\label{fbis.5}\begin{split}
    &\widehat{T}Mq(w',t)=\int e^{\frac{i}{h}\psi
      (w',t,y)}c(w',t,y;h)q(y)dy,\\ &(w',t)\in
    \mathrm{neigh\,}((0,1),{\bf C}^{n-1}\times {\bf C}),\end{split}
\end{equation}
where $c$ is an elliptic analytic symbol of order $0$ and 
$$
\psi (w',t,y)=\widetilde{\psi }(w',t,y)+\phi (y,t\eta _0')
$$
satisfies
\begin{equation}\label{fbis.6}
{{\psi }_\vert}_{y_n=0}=\widehat{\phi }(w',y'),
\end{equation}
\begin{equation}\label{fbis.7}
{{\partial _{y_n}\psi }_\vert}_{y_n=0}=i\left( r\left( y',0,\partial
  _{y'}\widehat{\phi }(w',y')-t\eta _0'
\right)^{\frac{1}{2}}+r(y',0,t\eta _0')^{\frac{1}{2}}\right).
\end{equation}

\par Assume for simplicity that $r(0,0,\eta _0')=1/4$. Then, at
the point $w'=0$, $t=1$, $y=0$, we have
$$
(\partial _{w'}\psi ,\partial _t\psi ,-\partial _{y'}\psi ,-\partial
_{y_n}\psi )=(0,0,0,-i),
$$
so $\kappa _{\widehat{T}M}(0,0,-i)=(0,1,0,0)$\footnote{We can verify
  directly that $\det \partial _{w',t}\partial _y\psi \ne 0$.}. Also, $\kappa
_{\widehat{T}M}=\kappa _{\widehat{T}}\circ \kappa _M$ and
\[
\begin{split}
& \kappa _M(0,0,-i)=(0,1,0,0),\\
&\kappa _{\widehat{T}}(0,1,0,0)=(0,1,0,0).
\end{split}
\]
Recall from (\ref{fop.21}) that
\[\begin{split}
  &\kappa _M:\, \mathrm{neigh\,}((0;0,-i),T^*\partial \Omega \times
  {\bf
    C}_{y_n^*}^-) \to\\
  &\mathrm{neigh\,}((0,1;0,0),{\bf R}_{x'}^{n-1}\times {\bf C}_t\times
  {\bf R}_{{x'}^*}^{n-1}\times \{t^*=0 \}),
\end{split}
\]
so
\[
\kappa _{\widehat{T}M}:\, \mathrm{neigh\,}((0,0,-i),T^*\partial \Omega
\times {\bf C}_{y_n^*}^-) \to \mathrm{neigh\,}((0,1,0,0),\Lambda
_{\widehat{\Phi }_0\oplus 0}).
\]
On the other hand, we have seen in Section \ref{fus} that $\kappa
_T(F)=\Lambda _{\Phi _1}$ and that the part $T^*\partial \Omega \times
{\bf C}_{y_n^*}^-$ of $F$ is mapped to $\Lambda _{\Phi
  _1^\mathrm{ext}}$. More locally,
\[
\begin{split}
&\kappa _T:\ \mathrm{neigh\,}((0,0,-i),T^*\partial \Omega \times {\bf
  C}_{y_n^*}^-) \to \mathrm{neigh\,}(\kappa _T(0,0,-i),\Lambda
_{\Phi _1^\mathrm{ext}})\\
&\kappa _S:\ \mathrm{neigh\,}(\kappa _T(0,0,-i),\Lambda
_{\Phi _1^\mathrm{ext}}) \to
\mathrm{neigh\,}((0,0,-i),T^*\partial \Omega \times {\bf
  C}_{y_n^*}^-).
\end{split}
\]
Using also (\ref{fop.21}), we get
\begin{equation}\label{fbi.19}
\kappa _{\widehat{T}MS}:\ \mathrm{neigh\,}(\pi _z\kappa
_T(0,0,-i),\Lambda _{\Phi _1^\mathrm{ext}})\to
\mathrm{neigh\,}((0,1,0,0),\Lambda _{\widehat{\Phi }_0\oplus 0}).
\end{equation}
We then also know that
$$\widehat{\Phi }_0(w')=\mathrm{vc}_{y,z}(-\Im \psi (w',t,y)+\Im \phi _T(z,y)).$$
This means (\cite{Sj82}) that the formal composition 
\begin{equation}\label{fbis.8}
\widehat{T}MSv(w',t)=\widetilde{C}h^{-\frac{n}{4}}\iint
e^{\frac{i}{h}(\psi (w',t,y)-\phi _T(z,y))}c(w',t,y;h)v(z)dzdy
\end{equation}
gives a well-defined operator
\begin{equation}\label{fbis.9}
\widehat{T}MS:\, H_{\Phi _1^\mathrm{ext},\pi _z\kappa _T(0,0,-i)}\to
H_{\widehat{\Phi }_0\oplus 0,(0,1)},
\end{equation}
that can be realized with the help of a good contour.

We shall next show that
\begin{equation}\label{fbis.10}
\widehat{T}Mu=(\widehat{T}MS)Tu \hbox{ in }H_{\widehat{\Phi }_0\oplus 0,(0,1)}
\end{equation} 
when $u$ is supported in $\{y_n\ge 0 \}$. The proof is the same as the
one for (\ref{fbi.4}). The right hand side in (\ref{fbis.10}) is equal
to 
$$
\mathrm{Const.}h^{-n}\iiint e^{\frac{i}{h}(\psi (w',t,x)-\phi
  _T(z,x)+\phi _T(z,y))}c(w',t,x;h) u(y) dydzdx,
$$ where the $y$-integration is over ${\bf R}^n_+$, and we may assume
without loss of generality, that $u$ has its support in a small
neighborhood of $y=0$. The $dzdx$ integration is, to start with, over the good
contour in (\ref{fbis.8}). This last integration can be viewed as
$T^\mathrm{t}S^\mathrm{t}$ acting on $e^{\frac{i}{h}\psi (w',t,\cdot
  )}c(w',t,\cdot ;h)$ and here $T^\mathrm{t}S^\mathrm{t}$ is the
identity operator, that can be realized with a good contour, so we get
$$
(\widehat{T}MS)Tu(w',t)=\int e^{\frac{i}{h}\psi (w',t,x)}c(w',t,x;h)u(x)dx=\widehat{T}Mu(w',t),
$$
and we have verified (\ref{fbis.10}).

\par Above, we have established (\ref{fbis.9}) as the quantum version
of (\ref{fbi.19}).
It follows by an easy adaptation of the exercise leading to
(\ref{fop.16}) that
\begin{equation}\label{fbi.22}\begin{split}
&\kappa _M(\mathrm{neigh\,}((0,0,-i),{\bf C}^{n-1}\times \{0 \} \times {\bf
  C}^-_{\eta _n}))\\ &=\mathrm{neigh\,}((0,0,1,0),{\bf
  C}^{n-1}_{x'}\times \{{x'}^*=0 \}\times
{\bf C}_t\times \{t^*=0 \}),
\end{split}
\end{equation}
and hence
\begin{equation}\label{fbi.20}
\kappa _{\widehat{T}MS}(\mathrm{neigh\,}(\kappa _T(0,0,-i),\Lambda
_{\Phi _3}))=\mathrm{neigh\,}((0,0,1,0),\Lambda _{0 \oplus 0}).
\end{equation}
The quantum version of (\ref{fbi.20}) is
\begin{equation}\label{fbi.21}
\widehat{T}MS:\, H_{\Phi _3,\pi _z\kappa _T(0,0,-i)}^\mathrm{loc}\to H_{0\oplus 0,(0,1)}^\mathrm{loc}.
\end{equation}

We also know that $\widehat{T}MS$ is an elliptic Fourier integral
operator. Consequently, (\ref{fbis.9}), (\ref{fbi.21}) have continuous
inverses. We also have the following result:
\begin{prop}\label{fbi1}
Assume that $u\in H^\mathrm{loc}_{\Phi _1,\pi _z\kappa
  _T(0,0,-i)}$ and that $\widehat{T}MSu\in H^\mathrm{loc}_{0\oplus
  0,(0,1)} $. Then $u\in H^\mathrm{loc}_{\Phi _3,\pi
  _z\kappa _T(0,0,-i)}$.
\end{prop}

\section{End of the proof of the main result}\label{ep}
\setcounter{equation}{0}
We will work with FBI and Laplace transforms of functions that are
independent of $h$ or that have some special $h$-dependence. Consider
a formal Fourier integral operator $u\mapsto Tu$, given by
\begin{equation}\label{ep.1}
Tu(x;h)=Ch^\alpha \int e^{\frac{i}{h}\phi (x,y)}u(y)dy,
\end{equation}
where $\phi =\phi _T$ is a quadratic form on ${\bf C}^{2n}_{x,y}$
satisfying 
\begin{equation}\label{ep.2}
\det \phi ''_{xy}\ne 0,
\end{equation}
and hence generating a canonical transformation, that will be used
below.
\begin{prop}\label{ep1}
If $u$ is independent of $h$, we have
\begin{equation}\label{ep.3}
(hD_h+\frac{1}{h}P_\alpha (x,hD;h))Tu=0,
\end{equation}
where
\begin{equation}\label{ep.4}
P_\alpha =p(x,hD)+ih\left(\alpha +\frac{1}{2}\mathrm{tr\,}(\phi
  ''_{xx}{\phi ''_{yx}}^{-1}\phi ''_{yy}{\phi ''_{xy}}^{-1} )\right),
\end{equation}
\begin{equation}\label{ep.5}
\begin{split}
  p(x,\xi )=&\frac{1}{2}\phi ''_{xx}x\cdot x+x\cdot (\xi -\phi
  ''_{xx}x)\\ &+\frac{1}{2}{\phi ''_{yx}}^{-1}\phi ''_{yy}{\phi
    ''_{xy}}^{-1}(\xi -\phi ''_{xx}x)\cdot (\xi -\phi ''_{xx}x)\\
  =&-\frac{1}{2}\phi ''_{xx}x\cdot x +\frac{1}{2}\phi ''_{xx}{\phi
    ''_{yx}}^{-1}\phi ''_{yy}{\phi ''_{xy}}^{-1}\phi ''_{xx}x\cdot
  x\\ & +x\cdot \xi -{\phi ''_{yx}}^{-1}\phi ''_{yy}{\phi ''_{xy}}^{-1}\phi
  ''_{xx}x\cdot \xi +\frac{1}{2}{\phi ''_{yx}}^{-1}\phi ''_{yy}{\phi
    ''_{xy}}^{-1}\xi \cdot \xi .
\end{split}
\end{equation}
\end{prop}
\begin{proof}
We have
\[
\begin{split}
hD_h\left(e^{\frac{i}{h}\phi (x,y)}
\right)&=-\frac{1}{h}e^{\frac{i}{h}\phi (x,y)},\\
hD_h\left(h^\alpha  \right)&=\frac{\alpha }{i}h^\alpha ,\\
hD_hTu(x;h)&=-\frac{1}{h}h^\alpha \int e^{\frac{i}{h}\phi
  (x,y)}(ih\alpha +\phi (x,y))u(y)dy,
\end{split}
\]
Try to write $\phi (x,y)=p(x,\phi '_x(x,y))$ for a suitable quadratic
form $p(x,\xi )$ (that will turn out to be the one given in
(\ref{ep.5})). We have
\begin{equation}\label{ep.6}
\phi (x,y)=\frac{1}{2}\phi ''_{xx}x\cdot x +\phi ''_{xy}y\cdot x
+\frac{1}{2}\phi ''_{yy}y\cdot y,
\end{equation}
\begin{equation}\label{ep.7}
\phi '_x=\phi ''_{xx}x+\phi ''_{xy}y,\hbox{ i.e. }y={\phi
  ''_{xy}}^{-1}(\phi '_x-\phi ''_{xx}x)
\end{equation}
and using the last relation from (\ref{ep.7}) in (\ref{ep.6}), we get 
\begin{equation}\label{ep.8}
\begin{split}
\phi (x,y)=&\frac{1}{2}\phi ''_{xx}x\cdot x+\phi ''_{yx}x\cdot {\phi
  ''_{xy}}^{-1}(\phi '_x-\phi ''_{xx}x)\\
&+\frac{1}{2}\phi ''_{yy}{\phi ''_{xy}}^{-1}(\phi '_x-\phi
''_{xx}x)\cdot {\phi ''_{xy}}^{-1}(\phi '_x-\phi ''_{xx}x), 
\end{split}
\end{equation}
where the $\phi ''_{yx}$ and ${\phi ''_{xy}}^{-1}$ in the second term
cancel, and we get $p(x,\phi '_x)$ with $p$ as in (\ref{ep.5}). 

\par To verify (\ref{ep.4}), it suffices to notice that
\[\begin{split}
e^{-\frac{i}{h}\phi (x,y)}p(x,hD_x)&\left(e^{\frac{i}{h}\phi
    (x,y)}\right)-p(x,\phi '_x)\\
&=\frac{1}{2}{\phi ''_{yx}}^{-1}\phi ''_{yy}{\phi
  ''_{xy}}^{-1}hD_x\cdot (\phi '_x)\\
&=\frac{1}{2}{\phi ''_{yx}}^{-1}\phi ''_{yy}{\phi
  ''_{xy}}^{-1}hD_x\cdot (\phi ''_{xx}x)\\
&=\frac{h}{2i}\phi ''_{xx}{\phi ''_{yx}}^{-1}\phi ''_{yy}{\phi
  ''_{xy}}^{-1}\partial _x\cdot x\\
&=\frac{h}{2i}\mathrm{tr\,}\left(\phi ''_{xx}{\phi ''_{yx}}^{-1}\phi ''_{yy}{\phi
  ''_{xy}}^{-1} \right).
\end{split}\]
\end{proof}
\begin{remark}\label{ep2}
{\rm
Let $\kappa _T:\, (y,-\phi '_y(x,y))\mapsto (x,\phi '_x(x,y))$ be the
canonical transformation associated to $T$ which can also be written
$$\kappa _T:\, (y,-(\phi ''_{yx}x+\phi ''_{yy}y)\mapsto (x,\phi
''_{xx}x+\phi ''_{xy}y),$$
or still $\kappa _T:\, (y,\eta )\mapsto (x,\xi )$ where
\[
\begin{split}
x&=-{\phi ''_{yx}}^{-1}(\eta +\phi ''_{yy}y)\\
\xi &=(\phi ''_{xy}-\phi ''_{xx}{\phi ''_{yx}}^{-1}\phi ''_{yy})y-\phi
''_{xx}{\phi ''_{yx}}^{-1}\eta .
\end{split}
\]
We see that the following three statements are equivalent:
\begin{itemize}
\item $\kappa _T$ maps the Lagrangian space ${\eta =0}$ to ${\xi =0}$.
\item $\phi ''_{xy}-\phi ''_{xx}{\phi ''_{yx}}^{-1}\phi ''_{yy}=0$.
\item $p(x,0)=0$, $p'_\xi (x,0)=0$, $\forall x$.
\end{itemize}
}
\end{remark}
\begin{ex}\label{ep3}
{\rm Consider
$$\widehat{T}{\cal L}u(x;h)=Ch^{\frac{1-n}{2}}\int
e^{\frac{i}{h}(\widehat{\phi } (x',y')+ix_ny_n)}u(y)dy,\ \phi =\phi
_{\widehat{T}}.$$ If $P'(x',hD_{x'};h)$ is the operator associated to
$\widehat{T}$ in $n-1$ variables, we get when $u$ is independent of
$h$,
\begin{equation}\label{ep.9}
(hD_h+\frac{1}{h}(P'(x',hD_{x'};h)+x_nhD_{x_n}))\widehat{T}{\cal L}u=0.
\end{equation}
Similarly (though not a direct consequence of Theorem \ref{ep1} but
rather of its method of proof) we have for ${\cal L}$ alone that
\begin{equation}\label{ep.9.5}
(hD_h+\frac{1}{h}x_nhD_{x_n}){\cal L}u=0.
\end{equation}
}
\end{ex}
\begin{ex}\label{ep4}
{\rm
Let $T$ be as above and assume that we are in the situation of Remark
\ref{ep2} so that $p(x,0)=0$, $p'_\xi (x,0)=0$. Then
$$
p(x,hD)=bhD\cdot hD
$$
where $b$ is a constant symmetric matrix. Then
$$
P_\alpha =p(x,hD)+ih(\alpha +f_0),\ f_0=\frac{n}{2}.
$$
and (\ref{ep.3}) reads
\begin{equation}\label{ep.10}
(hD_h+(hbD\cdot D+i(\alpha +f_0)))Tu=0.
\end{equation}
If $Tu=\sum_{m}^\infty h^kv_k\in H_0$ and $u$ is independent of $h$,
we can plug this expression into (\ref{ep.10}) and get the sequence of
equations,
\[
\begin{split}
&\left(\frac{m}{i} +i(\alpha +f_0) \right)v_m=0,\\
&\left(\frac{m+1}{i}+i(\alpha +f_0) \right)v_{m+1}+bD\cdot
Dv_{m}=0,\\
&\left(\frac{m+2}{i}+i(\alpha +f_0) \right)v_{m+2}+bD\cdot
Dv_{m+1}=0,\\
&...
\end{split}
\]
so unless $v_m\equiv 0$, we get $m=\alpha +f_0$. $v_m\in H_0$ can be
chosen arbitrary and $v_{m+1},\, v_{m+2},...$ are then uniquely determined.
}
\end{ex}

\par Now, consider the situation in Theorem \ref{int3} and let $q\in
L^\infty (\Omega )$ be independent of $h$ and such that $\sigma
_{\dot{{\Lambda}}}(y',t\eta _0')$ is a cl.a.s. on $\mathrm{neigh\,}(\{0
\}\times {\bf R}_+,{\bf R}^{n-1}\times {\bf R}_+)$ of order $-1$
(cf. (\ref{la.4}){),}
\begin{equation}\label{ep.11}
\sigma _{\dot{{\Lambda}}}(y',t\eta _0')\sim \sum_1^\infty n_k(y',t),
\end{equation}
where $n_k(y',t)$ is homogeneous of degree $-k$ in $t$.
\begin{equation}\label{ep.12}
|n_k(y',t)|\le C^{k+1}k^k|t|^{-k},\ y'\in \mathrm{neigh\,}(0,{\bf
  C}^{n-1}). 
\end{equation}
For the moment we shall only work with formal cl.a.s. and neglect
remainders in the asymptotic expansions. For the semi-classical symbol of
$\dot{{\Lambda}}$ we have
\begin{equation}\label{ep.13}\begin{split}
&\sigma _{\dot{{\Lambda}}}(y',t\eta _0'/h)\sim\sum_1^\infty
n_k(y',t/h)=\sum_1^\infty h^kn_k(y',t),\\ &(y',t)\in
\mathrm{neigh\,}((0,1),{\bf R}^{n-1}\times {\bf R}_+).
\end{split}
\end{equation}

\par Recall that $\sigma _{\dot{{\Lambda}}}(y',t\eta
_0'/h)=Mq(y',t;h)$. From (\ref{ep.13}) we infer that $\widehat{T}Mq$
is a cl.a.s. near $w'=0$, $t=1$:
\begin{equation}\label{ep.14}
\widehat{T}Mq\sim \sum_1^\infty h^km_k(w',t).
\end{equation}

Formally,
\begin{equation}\label{ep.15}
\widehat{T}M=(\widehat{T}M{\cal L}^{-1}){\cal L}.
\end{equation}
The canonical transformation $\kappa _{\cal L}$ is given by
$$
(y,\eta )\mapsto (y',i\eta _n,\eta ',iy_n).
$$
It maps the complex manifold $\eta '=0$, $y_n=0$ to the manifold $\{
(z,0) \}$ and the point $(0;0,-i)$ to $(0,1;0)$, so $\kappa _{{\cal
    \Lambda }^{-1}}=\kappa _{{\cal L}}^{-1}$ maps $\zeta =0$ to $\eta
'=0$, $y_n=0$ and we noticed in (\ref{fop.20.5})
(cf. (\ref{fop.19.4})), that $\kappa _M$ takes the complexified
conormal bundle to the zero section, and it maps
the point $(0;0,-i)$ to $(0,1;0)$. Thus $\kappa _{M{\cal L}^{-1}}$
maps the zero section $\zeta =0$ to the zero section and in particular $(0,1;0)$ to $(0,1;0)$. (We may
notice that this is global in the sense that we can extend $z_n$ to an
annulus, and we then get $t$ in an annulus.) Since $\kappa
_{\widehat{T}}$ maps the zero section to the zero section, we have the
same facts for $\kappa _{\widehat{T}M}$.

\par From the above, it is clear that $\widehat{T}M{\cal L}^{-1}$ maps
formal cl.a.s. to formal cl.a.s.

\par Recalling (\ref{ep.13}) for $\sigma _{\dot{{\Lambda}}}(y',t\eta
_0'/h)=Mq(y',t;h)$ and using that $\widehat{T}M{\cal L}^{-1}$ is an
elliptic Fourier integral operator whose canonical transformation maps
the zero-section to the zero-section, we see that there exists a
unique formal cl.a.s.
\begin{equation}\label{ep.18}
v\sim \sum_1^\infty v_k(z',z_n)h^k,\ z\in \mathrm{neigh\,}((0,1),{\bf
  C}^n),
\end{equation}
such that in the sense of formal stationary phase,
\begin{equation}\label{ep.19}
\widehat{T}Mq=\widehat{T}M{\cal L}^{-1}v.
\end{equation}
Now $q$ is independent of $h$, so $Mq$ satisfies a compatibility
equation of the form
\begin{equation}\label{ep.20}
\left( hD_h+\frac{1}{h}P_{\widehat{T}M}\right) Mq=0.
\end{equation}
This gives rise to a similar compatibility condition for $v$,
$$
\left( hD_h+\frac{1}{h}P_{{\cal L}M^{-1}\widehat{T}^{-1}\widehat{T}M}\right) v=0,
$$
or simply
$$\left(hD_h +\frac{1}{h}P_{\cal L} \right)v=0,$$
which is the same as (\ref{ep.9.5}):
\begin{equation}\label{ep.21}
(h\partial _h+z_n\partial _{z_n})v=0.
\end{equation}
Application of this to (\ref{ep.18}) gives
\begin{equation}\label{ep.22'}
(k+z_n\partial _{z_n})v_k=0,
\end{equation}
i.e.
\begin{equation}\label{ep.22}
v_k(z)=q_k(z')z_n^{-k},\ |q_k(z')|\le C^{k+1}k^k.
\end{equation}
Thus,
$$
v\sim \sum_1^\infty q_k(z')\left(\frac{h}{z_n} \right)^k
=\sum_0^\infty q_{k+1}(z')\left(\frac{h}{z_n} \right)^{k+1},
$$
and we see as in Section \ref{la} (with the difference that we now
deal with the semi-classical Laplace transform) that
\begin{equation}\label{ep.23}
v\sim {\cal L}{\widetilde{q}}(z;h),\ \widetilde{q}(y)=1_{[0,a]}(y_n)\sum
  _0^\infty \frac{q_{k+1}(y')}{k!}y_n^k,
\end{equation}
with $a>0$ small enough to ensure the convergence of the power series.

\par More precisely, (and now we end the limitation to formal symbols), as
in (\ref{fbis.10}), (\ref{fbi.4}), we check that 
\begin{equation}\label{ep.25}
\widehat{T}M\widetilde{q}\equiv (\widehat{T}M{\cal L}^{-1}){\cal
  L}\widetilde{q}\hbox{ in }H_{0,(0,1)}
\end{equation}
(up to an exponentially small error). By the construction of
$\widetilde{q}$, the right hand side is $\equiv \widehat{T}Mq$ in the
same space.

\par Put $r=q-\widetilde{q}$. Then 
\begin{equation}\label{ep.26}
\widehat{T}Mr\equiv 0\hbox{ in }H_{0,(0,1)}.
\end{equation}
Now, we replace ${\cal L}$ with $T$ and consider in the light of
(\ref{fbis.10}):
\begin{equation}\label{ep.27}
(\widehat{T}MS)Tr\equiv 0\hbox{ in }H_{0,(0,1)},
\end{equation}
which implies that $Tr\in H_{\Phi _1}$ satisfies
\begin{equation}\label{ep.28}
Tr\equiv 0\hbox{ in }H_{\Phi _1^\mathrm{ext},\pi _z\kappa _T(0;0,-i)}.
\end{equation}
As we saw in Section \ref{fus}, $\Lambda _{\Phi _1}$ contains the
closure $\overline{\Gamma }$ of the complex curve
$$
\Gamma =\kappa _T(\{(0;0,\eta _n);\, \Im \eta _n<0 \}),
$$
and $\kappa _T((0;0,-i))\in \Gamma $. Consequently, ${{\Phi
    _1}_\vert}_{\pi _z\Gamma }$ is harmonic and (\ref{ep.28}) and the
maximum principle imply that
\begin{equation}\label{ep.29}
Tr\equiv 0\hbox{ in }H_{\Phi _1,z},\ z\in \overline{\Gamma }.
\end{equation}
In particular,
\begin{equation}\label{ep.30}
Tr\equiv 0\hbox{ in }H_{\Phi _1,0}
\end{equation}
and a fortiori,
\begin{equation}\label{ep.31}
Tr\equiv 0\hbox{ in }H_{\Phi _0,0}.
\end{equation}
This implies that $r=0$ near
$y=0$. Hence $q=\widetilde{q}$ near $y=0$, which gives the theorem.
\section{Proof of Proposition \ref{int4}}\label{prp}
\setcounter{equation}{0}
We choose local coordinates $y=(y',y_n)$ as in the beginning of
Section \ref{la}. As in Proposition \ref{int4}, we assume that $q$ is
analytic in a neighborhood of $0$. { We shall first prove
  the analytic regularity of $\dot{\Lambda }(x',y')$ in $W'\times
  W'\setminus \mathrm{diag\,}(W'\times W')$, where $W'$ is a small
  neighborhood of $0$ in ${\bf R}^{n-1}$. Let $\chi \in C_0^\infty
  ({\bf R}^{n-1})$ has its support in a small neighborhood of $0$ and
  be equal to one in another such neighborhood. We may assume that the
  closure of $W'$ is contained in this second neighborhood. For
  $\alpha =(\alpha _{x'},\alpha _{\xi '})\in W'\times {\bf R}^{n-1}$,
  we put
\begin{equation}\label{ppr.1}
f_{\alpha }(x')=\chi (x')e^{\frac{i}{h}((x'-\alpha _{x'})\cdot \alpha
  _{\xi '}+\frac{i}{2}(x'-\alpha _{x'})^2)}.
\end{equation}
To prove the desired analytic regularity, it suffices (cf.\cite{Sj82})
to prove that for every $L\Subset W'\times W'\setminus
\mathrm{diag\,}(W'\times W')$,
\begin{equation}\label{ppr.2}
\iint \dot{\Lambda }(x',y')f_{\alpha }(x')f_{\beta }(y')dx'dy'={\cal
  O}(1)e^{-1/(Ch)},\end{equation}
for $(\alpha _{x'},\beta _{x'})\in L$, $\alpha _{\xi '},\beta _{\xi
  '}\in {\bf R}^{n-1}$, $1/2\le |(\alpha _{\xi '},\beta _{\xi '})|\le
2$, where $C=C_L>0$. The integral in (\ref{ppr.2}) is equal to 
\begin{equation}\label{ppr.3}
\int_{\Omega }q(z)K(f_{\alpha })(z) K(f_{\beta })(z) dz.
\end{equation}

\par If $|\alpha _{\xi '}|\ge 1/C$, then by analytic WKB (as we have
already used), we have up to an exponentially small error,
\begin{equation}\label{ppr.4}
K(f_\alpha )(z)=\widetilde{\chi }(z)b(z,\alpha ;h)e^{i\psi (z,\alpha )/h},
\end{equation}
where $\widetilde{\chi }\in C_0^\infty (\mathrm{neigh\,}(0,{\bf R}^n)$
is equal to one on $\overline{W'}\times \{z_n=0 \}$, $b$ is a
classical analytic symbol  of order $0$ and $\psi $ is the solution of
the eikonal equation
\begin{equation}\label{ppr.5}
\partial _{z_n}\psi =ir(z,\partial _{z'}\psi ),\ {{\psi
  }_\vert}_{z_n=0}=(z'-\alpha _{x'})\cdot \alpha _{\xi
  '}+\frac{i}{2}(z'-\alpha _{x'})^2. 
\end{equation}

\par If $|\alpha _{\xi '}|\ge 1/C$ and $|\beta _{\xi '}|\ge 1/C$, then
(\ref{ppr.4}) also holds with $\alpha $ replaced by $\beta $ and
(\ref{ppr.2}) follows, since $\alpha _{x'}\ne \beta _{x'}$ for
$(\alpha _{x'},\beta _{x'})\in L$ 

\par Recalling that, $1/2\le |(\alpha _{\xi '},\beta _{\xi '})|\le 2$,
it remains to discuss the case when one of $|\alpha _{\xi '}|$ and
$|\beta _{\xi '}|$ is $\ll 1$. We may assume that $|\alpha _{\xi
  '}|\ll 1$. By analytic regularity for elliptic boundary value
problems, we know that there exists $C>0$ such that $K(f_\alpha )$
extends to a holomorphic function of $z'$ in the domain
\begin{equation}\label{ppr.6}
\Re z'\in W',\ |\Im z'|\le 1/C,\ 0\le z_n\le 1/C,
\end{equation}
for all $\alpha \in {\bf R}^{2(n-1)}$. (This only uses that the
boundary data $f_\alpha $ has a holomorphic extension to a similar
domain and can be proved by tangential FBI-transforms combined with
complex WKB-constructions for Dirichlet problems with conditions at
$z_n=\mathrm{Const.}$.)

\par Up to an exponentially small error, the integral in (\ref{ppr.3})
is equal to
\begin{equation}\label{ppr.7}
\int_\Omega \widetilde{\chi }(z)q(z)K(f_\alpha )(z)b(z,\beta
;h)e^{i\psi (z,\beta )/h}dz,
\end{equation}
Thinking of $K(f_\alpha )(z)$ as low frequent in $z'$ since
$\alpha _{\xi '}$ is small, we make a contour deformation in $z'$:
\begin{equation}\label{ppr.8}
\Gamma :\ \mathrm{neigh\,}(\pi _{z'}\mathrm{supp\,}\widetilde{\chi
},{\bf R}^{n-1})\times \left[ 0,\frac{1}{C}\right]\ni z\mapsto (z'+i\delta
\widetilde{\chi }_1(z)\nu ,z_n),
\end{equation} 
where $0<\delta \ll 1$, $0\le \widetilde{\chi }_1\in C_0^\infty
\left(\mathrm{neigh\,}(\pi _{z'}\mathrm{supp\,}\widetilde{\chi },{\bf
  R}^{n-1})\times \left[ 0,\frac{1}{2C}\left[ \right)\right . \right .$ is equal to 1 on $
\mathrm{neigh\,}(\pi _{z'}\mathrm{supp\,}\widetilde{\chi },{\bf
  R}^{n-1})\times \left[0,\frac{1}{3C}\right]$ and we choose 
\begin{equation}\label{ppr.9}
\nu =\alpha _{\xi '}.
\end{equation}
$\widetilde{\chi }$ is still well-defined along $\Gamma $, snce the
deformation takes place inside the region where $\widetilde{\chi }=1$
and (\ref{ppr.7}) is equal to
\begin{equation}\label{ppr.10}
\int_\Gamma \widetilde{\chi }(z)q(z)K(f_\alpha )(z)b(z,\beta
;h)e^{i\psi (z,\beta )/h}dz.
\end{equation}
By the choice of $\Gamma $, we see that $be^{i\psi /h}={\cal
  O}(1)e^{-\delta /(C_0h)}$ along $\Gamma $, where $C_0$ is
independent of $\delta $. Now we can view $\Gamma $ as a global
complex deformation of $\Omega $ which coincides with $\Omega $
outside the support of $\widetilde{\chi }_1$ and on $\Gamma $ we can
consider $u=K(f_\alpha )$
as the solution of the Dirichlet problem
\begin{equation}\label{ppr.11}
(h^2\Delta -h^2V)u=0\hbox{ in }\Gamma ,\ {{u}_\vert}_{\partial \Gamma
}={{f_\alpha }_\vert}_{\partial \Gamma }.
\end{equation}
Now, 
$$
{{f_\alpha }_\vert}_{\partial \Gamma }={\cal O}(1)e^{C_0\delta |\alpha
_{\xi '}|/h},
$$
so by standard estimates for the elliptic and globally well-posed
Dirichlet problem (\ref{ppr.11}), we conclude that 
\begin{equation}\label{ppr.12}
{{K(f_\alpha )}_\vert}_{\Gamma }={\cal O}(1)e^{C_0\delta |\alpha _{\xi
  '}|/h}
\end{equation}
in the $L^2$-sense. It follows that the integral (\ref{ppr.10})
is 
$$
{\cal O}(1)e^{\delta (C_0|\alpha _{\xi '}|-1/C_0)/h},
$$
which is exponentially decaying, provided that $|\alpha _{\xi '}|\le
1/(2C_0^2)$. This concludes the proof of (\ref{ppr.2}) and hence of
the off diagonal analyticity of $\dot{\Lambda }(x',y')$.
}

\par
We {next prove the statement about the symbol and} adopt the alternative definition
of symbols in Remark \ref{int2.7}. It will also be convenient to
consider the semi-classical symbol of $\dot{{\Lambda}}$, $\sigma
_{\dot{{\Lambda}}}(y',\eta ';h)=\sigma _{\dot{{\Lambda}}}(y',\eta
'/h)$. For $y'\in \mathrm{neigh\,}(0,{\bf R}^{n-1})$,
\begin{equation}\label{prp.1}
\begin{split}
&\sigma _{\dot{{\Lambda}}}(y',\eta ';h)=\\
&-\partial _{y_n}GqK\left(\int \chi (t')e_{t'}(\cdot ;h)e^{i(\cdot )\cdot
  \eta '/h}dt'\right)(y',0)e^{-iy'\cdot \eta '/h},
\end{split}
\end{equation}
where $\chi $ and $e_t$ were defined in Remark \ref{int2.7} with $n$
there replaced by $n-1$. By analytic WKB (as we already used), we have
up to an exponentially small error,
\begin{equation}\label{prp.3}
K(e_{t'}(\cdot ;h)e^{i(\cdot )\cdot \eta
  '/h})=Ch^{\frac{1-n}{2}}a(y,\eta ';h)e^{i\phi (y,t,\eta ')/h},
\end{equation}
where $\phi $ is the solution of the eikonal problem
\begin{equation}\label{prp.4}
\partial _{y_n}\phi =ir(y,\partial _{y'}\phi )^{\frac{1}{2}},\ {{\phi
  }_\vert}_{y_n=0}=y'\cdot \eta '+\frac{i}{2}(y'-t)^2,
\end{equation}
and $a$ is an cl.a.s. of order 0, obtained from solving a sequence of
transport equations with the ``initial'' condition $a(y',0,\eta
';h)=1$.

Using again the analytic WKB-method we can find a cl.a.s. $b$ of order
0 in $h$ which solves the following inhomogeneous problem up to
exponentially small errors:
\[
\begin{cases}
(h^2\Delta -h^2V)(h^{\frac{1-n}{2}+3}b(y,t,\eta ';h)e^{\frac{i}{h}\phi
  (y,t,\eta ')}=Ch^{\frac{1-n}{2}+2}ae^{\frac{i}{h}\phi }q,\\
b(y',0,t,\eta ';h)=0.
\end{cases}
\]
Then up to exponentially small errors,
$$
GqK(e_t(\cdot ;h)e^{i(\cdot )\cdot \eta '/h})\equiv
h^{\frac{1-n}{2}+1}b(y,t,\eta ';h)e^{\frac{i}{h}\phi (y,t,\eta ')}
$$
and similarly for the gradients, so 
$$
-\left(\partial _{y_n}\right)_{y_n=0}GqK(e_t(\cdot ;h)e^{i(\cdot
  )\cdot \eta '/h})\equiv -h^{\frac{3-n}{2}}(\partial _{y_n}b)(y',0,t,\eta
';h)e^{\frac{i}{h}(y'\cdot \eta '+\frac{i}{2}(y'-t')^2)}.
$$
Multiplying with $\chi (t')$ and integrating in $t'$, we see that
$\sigma _{\dot{{\Lambda}}}(y',\eta ';h)$ is a cl.a.s. in the
semi-classical sense and this implies that $\sigma _{\dot{{\Lambda}}}(y',\eta)$
is a cl.a.s.

\end{document}